\documentclass{article}
\usepackage{amsmath}
\usepackage{amsfonts,amssymb,mathrsfs}
\usepackage{theorem}
\usepackage{hyperref}

\usepackage{tikz}
\usepackage{tikz-cd}
\usepgflibrary{arrows}
\usetikzlibrary{matrix}
\tikzset{>=stealth'} 

\newcommand{\Ob}{\mathrm{Ob}}
\newcommand{\lcm}{\mathrm{lcm}}

\input amssym.def

\numberwithin{equation}{section}

\numberwithin{figure}{section}

\makeatletter
\newcommand\qedsymbol{\hbox{$\Box$}}
\newcommand\qed{\relax\ifmmode\Box\else
  {\unskip\nobreak\hfil\penalty50\hskip1em\null\nobreak\hfil\qedsymbol
  \parfillskip=\z@\finalhyphendemerits=0\endgraf}\fi}

\newenvironment{proof-of}[1][{}]{\par\noindent \textbf{Proof of} {#1}. }{\qed}

\newenvironment{proof}[0]{\par\noindent \textbf{Proof}.}{\qed}

\DeclareMathOperator*{\Gal}{Gal}
\newcommand{\Dessin}{\mathsf{Dessin}}

\newcommand{\CP}{{\mathbb C}{\mathbb P}}

\newcommand{\GT}{\mathsf{GT}}
\newcommand{\GTSh}{\mathsf{GTSh}}
\newcommand{\GTh}{\widehat{\mathsf{GT}}}
\newcommand{\Zhat}{\widehat{\mathbb{Z}}}

\newcommand{\F}{\mathsf{F}}
\renewcommand{\P}{\mathsf{P}}
\newcommand{\Fh}{\widehat{\mathsf{F}}}

\newcommand{\PaB}{\mathsf{PaB}}
\newcommand{\NFI}{\mathsf{NFI}}

\newcommand{\PB}{\mathrm{PB}} 
\newcommand{\B}{\mathrm{B}} 
\newcommand{\N}{\mathsf{N}} 
\newcommand{\K}{\mathsf{K}} 
\renewcommand{\H}{\mathsf{H}}

\newcommand{\ct}{\mathfrak{ct}}
\newcommand{\fP}{\mathfrak{P}}

\newcommand{\Ih}{\mathrm{Ih}}
\newcommand{\A}{{\mathscr{A}}}

\newcommand{\PR}{{\mathscr{PR}}}

\newcommand{\Hom}{\mathrm{Hom}}
\newcommand{\Stab}{\mathrm{Stab}}

\newcommand{\Aut}{\mathrm{Aut}}

\newcommand{\ord}{\mathrm{ord}}
\newcommand{\isom}{\mathrm{isom}}

\newcommand{\tran}{\mathrm{tran}}

\newcommand{\conn}{{\mathrm {conn} }}

\newcommand{\id}{\mathrm{id}}

\newcommand{\tto}{\longrightarrow}
\newcommand{\iso}{\overset{\simeq}{\tto}}

\newcommand{\ti}[1]{{\tilde{#1}}}

\newcommand{\wh}[1]{{\widehat{#1}}}
\newcommand{\ol}[1]{{\overline{#1}}}

\newcommand{\txt}[1]{{\textrm{#1}}}
\newcommand{\e}[1]{{\textbf{#1}}}

\newcommand{\hs}{\heartsuit}

\newcommand{\al}{{\alpha}}

\newcommand{\ma}{\mathfrak{a}}

\newcommand{\ms}{{\mathfrak{s}}}

\newcommand{\si}{{\sigma}}
\newcommand{\ga}{{\gamma}}
\newcommand{\vf}{{\varphi}}

\newcommand{\ML}{{\mathcal{ML}}}

\newcommand{\cC}{\mathcal{C}}

\newcommand{\cD}{\mathcal{D}}
\newcommand{\cA}{\mathcal{A}}

\newcommand{\cI}{\mathcal{I}}
\newcommand{\cG}{\mathcal{G}}

\newcommand{\cZ}{{\mathcal{Z}}}
\newcommand{\cP}{\mathcal {P}}

\newcommand{\cO}{\mathcal{O}}

\newcommand{\hcP}{\widehat{\mathcal{P}}}

\newcommand{\bbP}{{\mathbb P}}

\newcommand{\bbZ}{{\mathbb Z}}
\newcommand{\bbQ}{{\mathbb Q}}

\newcommand{\te}{\theta}

\newcommand{\lan}{\langle\,}
\newcommand{\ran}{\,\rangle}

\date{}
\newtheorem{thm}{Theorem}[section]
\newtheorem{defi}[thm]{Definition}

\newtheorem{cor}[thm]{Corollary}

\newtheorem{prop}[thm]{Proposition}
\newtheorem{claim}[thm]{Claim}

\theorembodyfont{\rm}
\newtheorem{remark}[thm]{Remark}

\title{The action of $\GT$-shadows on child's drawings}

\author{Vasily A. Dolgushev}

\date{}

\begin{document}

\large

\maketitle

\begin{flushright}
{\small \it To the memory of Andrey Loginov}
\end{flushright}

\begin{abstract}
$\GT$-shadows \cite{GTshadows} are tantalizing objects that can be thought of as approximations
of elements of the mysterious Grothendieck-Teichmueller group $\GTh$ introduced by V. Drinfeld in 1990.
$\GT$-shadows form a groupoid $\GTSh$ whose objects are finite index subgroups of the 
pure braid group $\PB_4$, that are normal in $\B_4$. The goal of this 
paper is to describe the action of $\GT$-shadows on Grothendieck's child's drawings
and show that this action agrees with that of $\GTh$. We discuss the hierarchy of orbits 
of child's drawings with respect to the actions of $\GTSh$, $\GTh$, and the absolute Galois 
group $G_{\bbQ}$ of rationals. We prove that the monodromy group and the passport 
of a child's drawing are invariant with respect to the action of the subgroupoid $\GTSh^{\hs}$ of charming 
$\GT$-shadows.
We use the action of $\GT$-shadows 
on child's drawings to prove that every Abelian child's drawing admits 
a Belyi pair defined over $\bbQ$. Finally, we describe selected examples of non-Abelian 
child's drawings.
\end{abstract}

\tableofcontents

\section{Introduction}
\label{sec:intro}

The profinite version $\GTh$ of the Grothendieck-Teichmueller group \cite[Section 4]{Drinfeld} 
\cite{Leila-slides},  \cite{Leila-survey} connects topology to number theory in a fascinating way. 
$\GTh$ receives an injective homomorphism \cite{Ihara}, \cite[Section 3.2]{Leila-survey} from 
the absolute Galois group $G_{\bbQ}$ of rationals. It acts on Grothendieck's child's drawings 
and this action is compatible with the embedding $G_{\bbQ} \hookrightarrow \GTh$ and 
the standard action of $G_{\bbQ}$. 
 
Just like $G_{\bbQ}$, the group $\GTh$ is a rather intractable object. For example, currently, 
we know explicitly only two elements of $\GTh$: the identity element and the element that comes 
from the complex conjugation. The author also believes that, for an arbitrary child's drawing $\cD$, 
tools of modern mathematics do not allow us to say much about the orbit $\GTh(\cD)$. 

Let us denote by $\B_n$ (resp. $\PB_n$) the Artin braid group (resp. the pure braid group) on $n$ strands.
We denote by $x_{ij}$, $1 \le i < j \le n$ the standard generators of $\PB_n$ and recall that 
the elements $x_{12}, x_{23}$ generate a free subgroup of $\PB_3$. In this paper, 
we tacitly identify the free group $\F_2$ on two generators with 
$\lan x_{12}, x_{23} \ran \le \PB_3$. 

$\GTh$ can be defined as the group of continuous automorphisms\footnote{We tacitly assume that 
automorphisms of $\wh{\PaB}$ act trivially on the set of objects of $\wh{\PaB}(n)$ for every $n$.} 
of the profinite completion $\wh{\PaB}$ 
of the operad  $\PaB$ of parenthesized braids  \cite{BNGT},  \cite[Appendix A]{GTshadows}, \cite[Chapter 6]{Fresse1},  
\cite{Tamarkin}. $\PaB$ is an operad in the category of groupoids and it is ``assembled'' from the 
braid groups $(\B_{n})_{n \ge 2}$. Objects of the groupoid $\PaB(n)$ are completely parenthesized 
sequences of number $1,2,\dots, n$ in which each number appears exactly once. For example, $\Ob(\PaB(3))$
has $12$ objects: $(1,2)3,~1(2,3),~ (2,1)3,~ 2(1,3),~\dots$.
 
The isomorphisms $\al \in \PaB((1,2)3, 1(2,3))$ and $\beta \in \PaB((1,2), (2,1))$
shown in figure \ref{fig:beta-alpha} play an important role for $\PaB$. 
\begin{figure}[htp] 
\centering 
\begin{tikzpicture}[scale=1.5, > = stealth]
\tikzstyle{v} = [circle, draw, fill, minimum size=0, inner sep=1]
\draw (-0.7,0.5) node[anchor=center] {{$\beta ~ : = $}};
\node[v] (v1) at (0, 0) {};
\draw (0,-0.2) node[anchor=center] {{\small $1$}};
\draw (0,1.2) node[anchor=center] {{\small $2$}};
\node[v] (v2) at (1, 0) {};
\draw (1,-0.2) node[anchor=center] {{\small $2$}};
\draw (1,1.2) node[anchor=center] {{\small $1$}};
\node[v] (vv2) at (0, 1) {};
\node[v] (vv1) at (1, 1) {};
\draw [->] (v1) -- (vv1);
\draw (v2) -- (0.6, 0.4); 
\draw [->] (0.4, 0.6) -- (vv2);
\begin{scope}[shift={(3.5,0)}]
\draw (-0.7,0.5) node[anchor=center] {{$\al ~ : = $}};
\node[v] (v1) at (0, 0) {};
\draw (-0.15,-0.2) node[anchor=center] {{\small $($}};
\draw (0,-0.2) node[anchor=center] {{\small $1$}};
\node[v] (v2) at (0.5, 0) {};
\draw (0.5,-0.2) node[anchor=center] {{\small $2$}};
\draw (0.65,-0.2) node[anchor=center] {{\small $)$}};
\node[v] (v3) at (1.5, 0) {};
\draw (1.5,-0.2) node[anchor=center] {{\small $3$}};
\node[v] (vv1) at (0, 1) {};
\draw (0,1.2) node[anchor=center] {{\small $1$}};
\node[v] (vv2) at (1, 1) {};
\draw (0.85,1.2) node[anchor=center] {{\small $($}};
\draw (1,1.2) node[anchor=center] {{\small $2$}};
\node[v] (vv3) at (1.5, 1) {};
\draw (1.5,1.2) node[anchor=center] {{\small $3$}};
\draw (1.65,1.2) node[anchor=center] {{\small $)$}};
\draw [->] (v1) -- (vv1);  \draw [->] (v2) -- (vv2);  \draw [->] (v3) -- (vv3); 
\end{scope}
\end{tikzpicture}
\caption{The isomorphisms $\al$ and $\beta$} \label{fig:beta-alpha}
\end{figure}
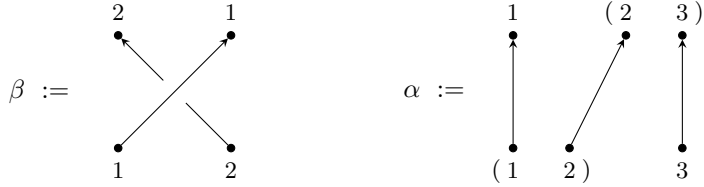

Due to \cite[Theorem 6.2.4]{Fresse1}, $\PaB$ is generated by $\al$ and $\beta$
(as the operad in the category of groupoids) and any relation involving $\al$ 
and $\beta$ is a consequence of the two hexagon relations and the pentagon relation
(see (A.13), (A.14) and (A.15) in \cite[Appendix A]{GTshadows}). 

Since $\al$ and $\beta$ are topological generators of $\wh{\PaB}$, every $\hat{T} \in \GTh$
is uniquely determined by the values $\hat{T}(\al) \in \wh{\PaB}( (1,2)3, 1(2,3) )$ and 
$\hat{T}(\beta) \in \wh{\PaB}((1,2), (2,1))$. 
In addition, since the automorphism group of $(1,2)3$ (resp. $(1,2)$) in $\wh{\PaB}$ is $\wh{\PB}_3$
(resp. $\wh{\PB}_2 \cong \Zhat$), every $\hat{T} \in \GTh$ is uniquely determined 
by a pair $(\hat{m}, \hat{f}) \in \Zhat \times \wh{\PB}_3$ via the equations
$$
\hat{T}(\beta) = \beta \circ x_{12}^{\hat{m}}, \qquad 
\hat{T}(\al) =  \al \circ \hat{f}. 
$$ 
Using the hexagon relations and
the pentagon relation, one can show that $\hat{f} \in \hat{\F}_2$. In fact, one can show that 
$\hat{f}$ belongs to the topological closure of the commutator subgroup $[ \hat{\F}_2,  \hat{\F}_2]$. 

It is easy to see \cite[Remark 1.1]{GTshadows} that the action of $\GTh$ on $\wh{\PB}_3 \cong \Aut((1,2)3)$ 
descends to the action on $\hat{\F}_2  \le \wh{\PB}_3$ and it is given (on the topological generators 
of $\Fh_2$) by the formulas 
\begin{equation}
\label{T-hat-F2-hat}
\hat{T}_{\Fh_2}(x):= x^{2\hat{m}+1}, \qquad
\hat{T}_{\Fh_2}(y):= \hat{f}^{-1} y^{2\hat{m}+1} \hat{f},
\end{equation}
where $(\hat{m}, \hat{f})$ is the pair in $\Zhat \times \wh{\F}_2$ corresponding to $\hat{T}$. 

Let $S_d$ be the symmetric group of degree $d$ and 
$\Hom_{\tran}(\Fh_2, S_d)$ be the set of continuous group homomorphisms 
$\psi: \Fh_2 \to S_d$ for which the subgroup $\psi(\Fh_2)$ acts transitively 
on the set $\{1,2,\dots, d\}$. Here $S_d$ is considered with the discrete topology. 

The set $\Hom_{\tran}(\Fh_2, S_d)$ carries the obvious action of $S_d$ by conjugation 
and a \e{child's drawing} of degree $d$ can be defined as an orbit of the $S_d$-action 
on $\Hom_{\tran}(\Fh_2, S_d)$. In these terms, it is easy to define the (right) action of $\GTh$ on 
child's drawings: given a child's drawing $[\psi]$ represented by a continuous group homomorphism 
$\psi: \Fh_2 \to S_d$ and $\hat{T} \in \GTh$ the child's drawing $[\psi]^{\hat{T}}$ is represented by 
the continuous group homomorphism 
$$
\psi \circ \hat{T}_{\Fh_2} : \Fh_2 \to S_d\,.
$$ 

Let us denote by $\NFI_{\PB_4}(\B_4)$ the poset of finite index normal subgroups $\N \unlhd \B_4$ 
satisfying the condition $\N \le \PB_4$. In paper \cite{GTshadows}, the authors introduced the 
concept of a $\GT$-\e{shadow}. Loosely speaking, a $\GT$-shadow is an onto morphism 
of (truncated) operads 
$$
\PaB^{\le 4} ~\to~ \PaB^{\le 4} /\sim_{\N}\,,
$$
where $\sim_{\N}$ is an equivalence relation on the set of morphisms of $\PaB^{\le 4}$ that 
comes from an element $\N$ of the poset $\NFI_{\PB_4}(\B_4)$. 

In paper \cite{GTshadows}, it was proved that $\GT$-shadows form a groupoid $\GTSh$
whose objects are elements of  $\NFI_{\PB_4}(\B_4)$. For every $\K, \N \in \NFI_{\PB_4}(\B_4)$, 
the set $\GTSh(\K, \N)$ may be identified with the set of isomorphisms of 
(truncated) operads 
$$
\PaB^{\le 4} /\sim_{\K} ~\iso~ \PaB^{\le 4} /\sim_{\N}\,.
$$

In this paper, we use the poset $\NFI_{\PB_4}(\B_4)$ and the set of all child's drawings 
to form a category $\Dessin$. We show that $\GTh$ acts naturally on the poset 
$\NFI_{\PB_4}(\B_4)$ and the action of $\GTh$ on child's drawings gives us a cofunctor 
$$
\A : \GTh_{\NFI} \to \Dessin
$$
from the corresponding transformation groupoid $\GTh_{\NFI}$ to the category $\Dessin$. 

We show (see Theorem \ref{thm:Action}) that $\GT$-shadows act on child's drawings in the sense 
that we have a natural cofunctor 
\begin{equation}
\label{A-sh-intro}
\A^{sh} : \GTSh \to \Dessin.
\end{equation}

Recall \cite[Section 2.4]{GTshadows} that, for every $\hat{T} \in \GTh$ and 
$\N \in \NFI_{\PB_4}(\B_4)$, we can produce a $\GT$-shadow $T_{\N}$ with 
the target $\N$, and $T_{\N}$ may be viewed as an approximation of $\hat{T}$. 
Using this passage from elements of $\GTh$ to $\GT$-shadows we define a functor 
\begin{equation}
\label{PR-intro}
\PR : \GTh_{\NFI} \to \GTSh\,.
\end{equation}

We prove (see Theorem \ref{thm:compatible}) that the action of $\GT$-shadows on child's drawings
is compatible with the action of $\GTh$ in the sense that the functors 
$$
\A \qquad \txt{and} \qquad \A^{sh} \circ \PR
$$
are equal in the strict sense. 

Using Theorem \ref{thm:compatible} and the compatibility of the $G_{\bbQ}$-action 
and the $\GTh$-action on child's drawings, we deduce several interesting corollaries: 

\begin{itemize}

\item Corollary \ref{cor:hierarchy} describes the hierarchy of orbits of 
the $\GTSh$-action, $\GTh$-action and $G_{\bbQ}$-action on the set of child's drawings.

\item Corollary \ref{cor:field-of-moduli} may be viewed as a version of \cite[Proposition 14]{HS}; 
this statement gives us a useful bound on the degree of the field of moduli of a child's drawing. 
 
\item Corollary \ref{cor:Galois-dessin-overQ} tells us that, using the poset 
$\NFI_{\PB_4}(\B_4)$, we can produce many examples of Galois child's drawings 
that admit Belyi pairs defined over $\bbQ$.

\end{itemize}
 
Using consequences of the hexagon relations for $\GT$-shadows, we show 
(see Theorem \ref{thm:passport-invar})
that the passport of a child's drawing is invariant with respect to the action of 
(charming) $\GT$-shadows. 
 
In this paper, we also prove some basic facts about Abelian child's drawings and show 
(see Corollary \ref{cor:GQ-orbit-singleton} and \cite[Corollary 3.5]{Hidalgo}) 
that every Abelian child's drawing admits a Belyi pair defined over $\bbQ$.

Finally, we describe selected examples of non-Abelian child's drawings whose $\GTSh$-orbits 
were computed using software package \cite{Package-GT}. Whenever possible, the results were 
compared to $G_{\bbQ}$-orbits from database \cite{Belyi_database}. It is amazing to see that, if 
a $\GTSh$-orbit and the $G_{\bbQ}$-orbit of a child's drawing can be computed then these orbits 
coincide. 

\subsection*{Organization of the paper}
Section \ref{sec:background} is devoted to the background material. It contains a brief reminder 
of the groupoid $\GTSh$ and its link to $\GTh$. In this section, we also recall child's drawings, 
introduce the category $\Dessin$ and define the action of $\GTh$ on child's drawings in terms 
of a cofunctor from a certain transformation groupoid $\GTh_{\NFI}$ to $\Dessin$.
 
In Section \ref{sec:GTshadows-act}, we define the action of $\GT$-shadows on child's drawings, 
describe the relationship between orbits of a child's drawing with respect to different actions, and 
prove that the monodromy group and the passport of a child's drawing are invariant with respect to 
the action of (charming) $\GT$-shadows.   
 
Section \ref{sec:Abelian} is devoted to various properties of Abelian child's drawings. 
In this section, we prove that every Abelian child's drawing admits a 
Belyi pair defined over $\bbQ$. 

In Section \ref{sec:examples}, we describe selected examples of non-Abelian child's drawings 
whose $\GTSh$-orbits were computed. 

In Appendix \ref{app:simple-hexa}, we prove that charming $\GT$-shadows satisfy 
versions of relations (4.3) and (4.4) from \cite[Section 4]{Drinfeld}.

In Appendix \ref{app:package}, we give an outline of the package $\GT$ for 
working with $\GT$-shadows and their action on child's drawings. For more details, 
please see the documentation \cite{Package-GT} for this package.

\begin{remark}  
\label{rem:GTh}
We should remark that the notation $\GTh$ is a bit misleading: the hat over ``$\GT$'' 
does not mean that $\GTh$ is a profinite completion of some ``well known group $\GT$''. 
However, please see \cite[Theorem 3.8]{GTshadows}.
\end{remark}

~\\
{\bf Acknowledgement.} V.A.D. is thankful to John Voight for many illuminating discussions 
and for his patience with his (possibly naive) questions. V.A.D. is thankful to Leila Schneps for 
her unbounded enthusiasm about $\GT$-shadows and everything related to $\GT$-shadows. 
V.A.D. is thankful to Benjamin Enriquez for his suggestion to use the transformation groupoid 
corresponding to the action of $\GTh$ on $\NFI_{\PB_4}(\B_4)$.
V.A.D. benefitted from discussions with David Harbater, Julia Hartmann, Sean O'Donnell and Florian Pop. 
V.A.D. is thankful to Khanh Le, Aidan Lorenz, Chelsea Zackey for discussions and their contributions 
to the software package $\GT$. V.A.D. is thankful to Pavol Severa for showing him 
the works of Pierre Guillot. This paper benefitted from many useful remarks and suggestions of 
the anonymous referee and V.A.D. would like to acknowledge the referee for her/his diligence.
V.A.D. acknowledges Temple University for 2021 Summer Research Award.

~\\
{\bf Memorial note.} Andrey Loginov (1977-2016) was an experimental physicist,
and his work helped to push frontiers of knowledge related to the fundamental questions 
about elementary particles.  There were several moments in my life when he played for me
a role of an older brother. Life connected us in bizarre ways by cities scattered around the globe: 
Berdychiv, Chicago, Dolgoprudny, Moscow and Tomsk.  It is very sad that I did not have 
a chance to say a proper goodbye to him...

\subsection{Notational conventions}
For a set $X$ with an equivalence relation and $a \in X$ we will denote by $[a]$
the equivalence class represented by the element $a$. 

The notation $\B_n$ (resp. $\PB_n$) is reserved for the Artin braid group 
on $n$ strands (resp. the pure braid group on $n$ strands). The standard generators 
of $\B_n$ are denoted by $\si_1, \dots, \si_{n-1}$ and the standard generators of 
$\PB_n$ are denote by $x_{ij}$ for $1 \le i < j \le n$. We set
$$
c := x_{23} x_{12} x_{13}  \, \in \, \PB_3
$$
and recall that $c$ generates the center $\cZ(\PB_3)$ of $\PB_3$ 
(as well as the center $\cZ(\B_3)$ of $\B_3$).  
The free group $\F_2$ on two generators is tacitly identified with the subgroup 
$\lan x_{12}, x_{23} \ran \le \PB_3$. Occasionally, we denote the standard 
generators of $\F_2$ by $x$ and $y$, i.e. $x:=x_{12}$ and $y: = x_{23}$.
 
$S_d$ denotes the group of bijections $\{1,2, \dots, d\} \iso \{1,2, \dots, d\}$ 
(i.e. the symmetric group of degree $d$). 
A subgroup $H \le S_d$ is called \e{transitive}
if its standard action on $\{1,2, \dots, d\}$ is transitive. 
Every finite group is tacitly considered with the discrete topology. For a group $G$, the notation 
$[G, G]$ is reserved for the commutator subgroup of $G$. 
For a normal subgroup $H\unlhd G$ of finite index, we denote by $\NFI_{H}(G)$ the poset 
of finite index normal subgroups $N$ in $G$ such that $N\le H$. The notation $\NFI(G)$ is 
reserved for the poset of finite index normal subgroups of $G$. 
For $\N \in \NFI(G)$, $\cP_{\N}$ denotes the standard projection 
$$
\cP_{\N} : G \to G/\N, \qquad 
\cP_{\N}(g):= g \N.
$$
Moreover, $\hcP_{\N}$ denotes the standard continuous group homomorphism from the 
profinite completion $\wh{G}$ of $G$ to the finite group $G/\N$. 
For a group $G$ and $g \in G$, 
the notation $\cZ_G(g)$ is reserved for the centralizer of $g$ in $G$. 
If a group $G$ is residually finite (e.g. $G = \F_n, ~\PB_n$ or $G = \B_n$), then we 
tacitly identify $G$ with its image in the profinite completion $\wh{G}$.

For objects $a,b$ of a category $\cC$, $\cC(a,b)$ often denotes the set of morphisms 
from $a$ to $b$. For a groupoid $\cG$, the notation $\ga \in \cG$ means that 
$\ga$ is a \e{morphism} of this groupoid. We assume that all functors are covariant and 
the word ``cofunctor'' means a contravariant functor. 

We will freely use the language of operads \cite[Section 3]{notes}, \cite[Chapter 1]{Fresse1}, 
\cite{LV-operads}, \cite{Markl-Stasheff}, \cite{Jim-operad}. In this paper, we encounter operads 
in the category of sets and in the category of (topological) groupoids. The category of topological 
groupoids is treated in the ``strict sense''. For example, the associativity axioms for 
the {\it elementary insertions}\footnote{In the literature, elementary insertions are sometimes called {\it partial 
compositions.}} $\circ_i$ (for operads in the category of groupoids)
are satisfied ``on the nose''.   

For an integer $q \ge 1$, a $q$-\e{truncated operad} in the category of groupoids is 
a collection of groupoids $\{ \cG(n) \}_{1 \le n \le q}$ such that 
\begin{itemize}

\item For every $1 \le n \le q$, the groupoid $\cG(n)$ is equipped with an action of $S_n$.

\item For every triple of integers $i,n,m$ such that $1 \le i  \le n$, $n,m, n+m -1 \le q$
we have functors 
\begin{equation}
\label{circ-i-fun}
\circ_i :  \cG(n) \times  \cG(m) \to  \cG(n+m -1).
\end{equation}

\item The axioms of the operad for $\{ \cG(n) \}_{1 \le n \le q}$ are satisfied 
in the cases where all the arities are $\le q$. 

\end{itemize}

For every operad $\cO$ and for every integer $q \ge 1$, the disjoint union 
$$
\cO^{\le q} : = \bigsqcup_{n=0}^q \cO(n)
$$
is clearly a $q$-truncated operad. 

In this paper, we mostly consider $4$-truncated operads. So, we will simply call them \e{truncated operads}. 
  
The operad $\PaB$ of parenthesized braids, its truncation $\PaB^{\le 4}$ and 
its completion $\wh{\PaB}^{\le 4}$ play an important role in this paper.
See \cite{BNGT},  \cite[Appendix A]{GTshadows}, \cite[Chapter 6]{Fresse1},  
\cite{Tamarkin} for more details about these objects. 

\section{Background material}
\label{sec:background}

\subsection{Reminder of $\GT$-shadows}
\label{sec:GTsh-reminder}

In this section, we review the groupoid $\GTSh$ whose objects are elements of the poset 
$\NFI_{\PB_4}(\B_4)$ and whose morphisms are called $\GT$-shadows. 

Recall \cite[Section 2.2]{GTshadows} that every element $\N \in \NFI_{\PB_4}(\B_4)$ gives 
us a compatible equivalence relation $\sim_{\N}$ on the truncation 
\begin{equation}
\label{PaB-le-4}
\PaB^{\le 4} := \PaB(1) \sqcup \PaB(2) \sqcup \PaB(3) \sqcup \PaB(4)
\end{equation}
of the operad $\PaB$, i.e. 
\begin{itemize}

\item $\sim_{\N}$ is an equivalence relation on the set of morphisms of 
$\PaB(1) \sqcup \PaB(2) \sqcup \PaB(3) \sqcup \PaB(4)$;

\item $\sim$ is compatible with the actions of the symmetric groups $S_2, S_3, S_4$; 

\item $\sim$ is compatible with the composition of morphisms and the operadic insertions;

\item finally, the quotient operad $\PaB^{\le 4}/\sim$ is finite.

\end{itemize}

More precisely, given $\N \in \NFI_{\PB_4}(\B_4)$, we produce $\N_{\PB_3} \in \NFI_{\PB_3}(\B_3)$ and  
$\N_{\PB_2} \in \NFI_{\PB_2}(\B_2)$. (Since $\PB_2$ is the infinite cyclic group and $\B_2$
is Abelian, $\N_{\PB_2}$ is uniquely determined by its index  $N_{\ord} :=|\PB_2 : \N_{\PB_2} |$).
Then $\N$ (resp. $\N_{\PB_3}$,  $\N_{\PB_2}$) gives us an equivalence relation 
on $\PaB(4)$ (resp. $\PaB(3)$, $\PaB(2)$). Due to \cite[Proposition 2.4]{GTshadows}, these 
equivalence relations on $\PaB(4)$, $\PaB(3)$ and $\PaB(2)$ assemble into a compatible 
equivalence relation\footnote{Note that $\PaB(1)$ is the groupoid with exactly one object and exactly 
one (identity) morphism.} $\sim_{\N}$ on \eqref{PaB-le-4}. 

For  $\N \in \NFI_{\PB_4}(\B_4)$, we denote by $\cP_{\N}$ the standard projection 
\begin{equation}
\label{cP-N}
\cP_{\N} : \PaB^{\le 4} \to \PaB^{\le 4}/ \sim_{\N}. 
\end{equation}
Similarly, $\hcP_{\N}$ denotes the standard continuous (onto) map of truncated operads 
\begin{equation}
\label{hcP-N}
\hcP_{\N} : \wh{\PaB}^{\le 4} \to \PaB^{\le 4}/ \sim_{\N}, 
\end{equation}
where $\wh{\PaB}^{\le 4}$ is the profinite completion of $\PaB^{\le 4}$. 

We set
\begin{equation}
\label{N-F-2}
\N_{\F_2} := \N_{\PB_3} \cap \F_2\,,
\end{equation}
where $\F_2$ is (tacitly) identified with the subgroup $\lan x_{12}, x_{23} \ran \le \PB_3$. 
 
A $\GT$-\emph{pair} with the target $\N$ is a morphism of truncated operads 
$\PaB^{\le 4} \to \PaB^{\le 4}/\sim_{\N}$ and such morphisms are in 
bijection with pairs 
\begin{equation}
\label{pair}
(m, f \N_{\PB_3} ) \in \{0,1, \dots, N_{\ord}-1\} \times  \PB_3 / \N_{\PB_3} 
\end{equation}
satisfying the hexagon relations
\begin{equation}
\label{hexa1}
\si_1 x_{12}^m \, f^{-1} \si_2 x_{23}^m f \, \N_{\PB_3} ~ = ~ 
f^{-1} \si_1 \si_2 x_{12}^{-m} c^m \, \N_{\PB_3}\,,
\end{equation} 
\begin{equation}
\label{hexa11}
f^{-1} \si_2 x_{23}^m f \, \si_1 x_{12}^m \,  \N_{\PB_3}  ~=~ 
\si_2 \si_1 x_{23}^{-m} c^m \, f \,  \N_{\PB_3}
\end{equation}
and the pentagon relation 
\begin{equation}
\label{GT-penta}
\vf_{234}(f)\, \vf_{1,23,4} (f)\, \vf_{123}(f) \, \N    ~ = ~  \vf_{1,2,34}(f)   \vf_{12,3,4}(f)\, \N.
\end{equation}
Both sides of \eqref{hexa1} and \eqref{hexa11} are elements of $\B_3 / \N_{\PB_3}$ 
and both sides of \eqref{GT-penta} are elements of $\PB_4/\N$. Explicit formulas 
for the group homomorphisms
$\vf_{234}, \vf_{1,23,4}, \vf_{123}, \vf_{1,2,34}$ and $\vf_{12,3,4}$ from $\PB_3$ to 
$\PB_4$ are given in \cite[Appendix A.4, (A.18)]{GTshadows}.

Just as in \cite{GTshadows}, we tacitly identify morphisms $\PaB^{\le 4} \to \PaB^{\le 4}/\sim_{\N}$
with pairs \eqref{pair} satisfying \eqref{hexa1}, \eqref{hexa11} and \eqref{GT-penta}.

It is convenient to represent $\GT$-pairs by tuples $(m,f) \in \bbZ \times \PB_3$ and 
we denote by $[m,f]$ the $\GT$-pair represented by the tuple $(m,f)$. 
For a $\GT$-pair $[m,f]$ with the target $\N$, we denote by 
\begin{equation}
\label{T-m-f}
T_{m,f} : \PaB^{\le 4} \to \PaB^{\le 4}/\sim_{\N}
\end{equation}
the corresponding morphism of (truncated) operads. 

It is clear that, for every $\GT$-pair $[m,f]$ with the target $\N$,
$T_{m,f}$ gives us the following group homomorphisms  
$$
T^{\PB_4}_{m,f} : \PB_4 \to \PB_4/\N,  \qquad
T^{\PB_3}_{m,f} : \PB_3 \to \PB_3/\N_{\PB_3}\,,  \qquad
T^{\PB_2}_{m,f} : \PB_2 \to \PB_2/\N_{\PB_2}\,.  
$$
For more details, see \cite[Corollary 2.7]{GTshadows}.

A $\GT$-\emph{shadow} with the target $\N$ is an \e{onto} morphism of truncated operads 
$\PaB^{\le 4} \to \PaB^{\le 4}/\sim_{\N}$ and such morphisms are in bijection with 
pairs \eqref{pair} that satisfy the following conditions: 
\begin{itemize}

\item $(m,f)$ obeys relations \eqref{hexa1}, \eqref{hexa11}, \eqref{GT-penta},

\item $2m+1$ represents a unit in the ring $\bbZ/ N_{\ord} \bbZ$, and

\item the group homomorphism $T^{\PB_3}_{m,f}: \PB_3 \to  \PB_3/\N_{\PB_3}$ is onto. 

\end{itemize}
We will identify $\GT$-shadows (with the target $\N$) with pairs \eqref{pair} satisfying 
the above conditions and denote by $\GT(\N)$ the set of $\GT$-shadows with the target $\N$.  

Let $\N \in \NFI_{\PB_4}(\B_4)$ and $[m,f] \in \GT(\N)$. 
Due to \cite[Proposition 2.11]{GTshadows}, the ``kernel'' of the 
morphism $T_{m,f}$ \eqref{T-m-f} is a compatible equivalence relation $\sim_{\N^{\ms}}$, where 
$$
\N^{\ms} := \ker(T^{\PB_4}_{m,f}).
$$
Hence $T_{m,f}$ induces an isomorphism of truncated operads 
\begin{equation}
\label{T-m-f-isom}
T^{\isom}_{m,f} : \PaB^{\le 4}/\sim_{\N^{\ms}} ~\iso~ \PaB^{\le 4}/\sim_{\N}
\end{equation}
and factors as follows: $T_{m, f} = T^{\isom}_{m,f} \circ \cP_{\N^{\ms}}$.

Thus $\GT$-shadows form a groupoid $\GTSh$: objects of $\GTSh$ are elements of
$\NFI_{\PB_4}(\B_4)$; for $\N^{(1)}, \N^{(2)} \in \NFI_{\PB_4}(\B_4)$, 
\begin{equation}
\label{GTSh-morph}
\GTSh(\N^{(2)}, \N^{(1)}) := \{[m,f] \in \GT(\N^{(1)}) ~|~ \N^{(2)} = \ker(T^{\PB_4}_{m,f}) \}
\end{equation}
or equivalently $\GTSh(\N^{(2)}, \N^{(1)})$ consists of isomorphisms of 
(truncated) operads 
$$
\PaB^{\le 4}/\sim_{\N^{(2)}} ~\iso~ \PaB^{\le 4}/\sim_{\N^{(1)}}.
$$
The composition of morphisms is defined via the identification 
of elements $[m,f] \in \GT(\N)$ with isomorphisms \eqref{T-m-f-isom}. 
(See \eqref{compose} below for the explicit formula of the composition of (practical) $\GT$-shadows
in terms of their representatives.)

\bigskip

Recall \cite[Corollary 2.8]{GTshadows} that, for every $[m,f] \in \GT(\N)$, the homomorphism 
$T^{\PB_3}_{m,f} : \PB_3 \to \PB_3/\N_{\PB_3}$ is given by the explicit formulas: 
\begin{equation}
\label{T-PB-3}
T^{\PB_3}_{m,f}(x_{12}) = x^{2m+1}_{12} \N_{\PB_3}, \quad  
T^{\PB_3}_{m,f}(x_{23}) =  f^{-1} x^{2m+1}_{23} f \N_{\PB_3}, 
\quad 
T^{\PB_3}_{m,f}(c) = c^{2m+1} \N_{\PB_3}.
\end{equation}

Since $\PB_3 \cong \lan x_{12}, x_{23} \ran \times \lan c \ran$, the restriction 
of $T^{\PB_3}_{m,f}$ to $\F_2 = \lan x_{12}, x_{23} \ran $ gives us a group homomorphism 
\begin{equation}
\label{T-F-2}
T^{\F_2}_{m,f} : \F_2 \to \F_2/\N_{\F_2}
\end{equation}
with 
\begin{equation}
\label{T-F-2-def}
T^{\F_2}_{m,f}(x) := x^{2m+1}\, \N_{\F_2}\,, \qquad  
T^{\F_2}_{m,f}(y) :=  f^{-1} y^{2m+1} f\, \N_{\F_2}\,.
\end{equation}

Let us prove the following useful statement:
\begin{prop}  
\label{prop:T-F-2-onto}
For every $\N \in \NFI_{\PB_4}(\B_4)$ and every $[m,f] \in \GT(\N)$, the group 
homomorphism \eqref{T-F-2} is onto.
\end{prop}  
\begin{proof}
Due to \cite[Proposition 2.3]{GTshadows}, $N_{\ord}$ is the least common multiple
of the orders of the elements $x_{12} \N_{\PB_3},~ x_{23} \N_{\PB_3}$ and $c \N_{\PB_3}$ 
in $\PB_3/ \N_{\PB_3}$.  

Therefore, since $2m+1$ is coprime with $N_{\ord}$, there exist $q_1, q_2 \in \bbZ$ such that 
\begin{equation}
\label{progress}
T^{\F_2}_{m,f} (x^{q_1}_{12}) = x_{12} \N_{\F_2}, 
\qquad 
T^{\F_2}_{m,f} (x^{q_2}_{23}) = f^{-1} x_{23} f \N_{\F_2}. 
\end{equation}

Since the homomorphism $T^{\PB_3}_{m,f}$ is onto, there exists $w \in \PB_3$ such that 
\begin{equation}
\label{fN-from-PB-3}
T^{\PB_3}_{m,f}(w) = f \N_{\PB_3}
\end{equation}

Since $\PB_3 = \lan x_{12}, x_{23} \ran \times \lan c \ran$,
$$
w = w_1 c^k\,,
$$
where $w_1 \in \F_2$. Therefore, using \eqref{fN-from-PB-3}, we get 
$$
T^{\F_2}_{m,f} (w_1 x^{q_2}_{23} w_1^{-1}) = 
T^{\PB_3}_{m,f} (w_1 c^k x^{q_2}_{23} c^{-k} w_1^{-1} ) = 
$$
$$
T^{\PB_3}_{m,f} (w x^{q_2}_{23}  w^{-1} ) = 
f\, (f^{-1} x_{23} f ) \, f^{-1} \N_{\PB_3} = x_{23} \N_{\PB_3}. 
$$
Thus, $T^{\F_2}_{m,f} (w_1 x^{q_2}_{23} w_1^{-1}) = x_{23} \N_{\F_2}$. 

Since both generators $x_{12} \N_{\F_2}$ and $x_{23} \N_{\F_2}$ of $\F_2 /\N_{\F_2}$ 
belong to $T^{\F_2}_{m,f}(\F_2)$, we proved that $T^{\F_2}_{m,f}$ is indeed onto.
\end{proof}

\subsubsection{$\GTh$ versus the groupoid $\GTSh$}
\label{sec:GTh-v-GTSh}

Just as in \cite{GTshadows}, we denote by $\cI$ the standard 
morphism $\PaB^{\le 4} \to \wh{\PaB}^{\le 4}$. 

Let $\hat{T} \in \GTh$ and $\N \in  \NFI_{\PB_4}(\B_4)$. It was shown in 
\cite[Section 2.4]{GTshadows} that the formula
\begin{equation}
\label{T-N}
T_{\N} := \hcP_{\N} \circ \hat{T} \circ  \cI : \PaB^{\le 4} \to \PaB^{\le 4} / \sim_{\N}
\end{equation}
defines a $\GT$-shadow with the target $\N$.

Let us prove that 
\begin{prop}  
\label{prop:GThat-NFI}
The assignment 
\begin{equation}
\label{That-acts}
\N^{\hat{T}} : =  \ker\big(\PB_4 \overset{T^{\PB_4}_{\N}}{\tto} \PB_4/\N \big)
\end{equation}
defines a right action of $\GTh$ on the set $\NFI_{\PB_4}(\B_4)$. 
Moreover, the assignment $\hat{T} \mapsto T_{\N}$ defines a functor $\PR$ from 
the corresponding transformation groupoid to $\GTSh$.
\end{prop}  
\begin{proof}
It is clear that, if $\hat{T}$ is the identity element of $\GTh$, then 
$$
T_{\N} = \cP_{\N} : \PaB^{\le 4} \to \PaB^{\le 4} / \sim_{\N}.
$$

For every $\hat{T} \in \GTh$, we have the following commutative diagram 
of maps of truncated operads:
\begin{equation}
\label{diag-T-N}
\begin{tikzpicture}
\matrix (m) [matrix of math nodes, row sep=1.5em, column sep=2.1em]
{\wh{\PaB}^{\le 4} & \wh{\PaB}^{\le 4} \\
 \PaB^{\le 4}/ \sim_{\N^{\hat{T}}} ~ &  ~ \PaB^{\le 4}/ \sim_{\N}\,. \\};
\path[->, font=\scriptsize]
(m-1-1) edge node[above] {$\hat{T}$} (m-1-2) edge node[left] {$\hcP_{\N^{\hat{T}}}$} (m-2-1)
(m-2-1) edge node[above] {$T^{\isom}_{\N}$} (m-2-2) 
(m-1-2) edge  node[right] {$\hcP_{\N}$} (m-2-2);
\end{tikzpicture}
\end{equation}

Let $\hat{T}_1, \hat{T}_2 \in \GTh$, $\N \in \NFI_{\PB_4}(\B_4)$, 
$\N^{(1)} : = \N^{\hat{T}_1}$ and $\N^{\ms} : = (\N^{(1)})^{\hat{T}_2}$. 
Combining the corresponding commutative diagrams for $\hat{T}_1$ and $\hat{T}_2$,  
we get the commutative diagram: 
\begin{equation}
\label{diag-T2-T1-N}
\begin{tikzpicture}
\matrix (m) [matrix of math nodes, row sep=2.3em, column sep=2.3em]
{ ~~ & \wh{\PaB}^{\le 4} & \wh{\PaB}^{\le 4} &  \wh{\PaB}^{\le 4} \\
\PaB^{\le 4}  & \PaB^{\le 4}/ \sim_{\N^{\ms}}  &  \PaB^{\le 4}/ \sim_{\N^{(1)}} &   \PaB^{\le 4}/ \sim_{\N} \\};
\path[->, font=\scriptsize]
(m-2-1) edge node[above] {$\cP_{\N^{\ms}}$} (m-2-2)
 edge node[above] {$\cI$} (m-1-2)
(m-1-2) edge node[above] {$\hat{T}_2$} (m-1-3) edge node[left] {$\hcP_{\N^{\ms}}$} (m-2-2)
(m-1-3) edge node[above] {$\hat{T}_1$} (m-1-4) edge node[right] {$\hcP_{\N^{(1)}}$} (m-2-3)
(m-1-4) edge node[right] {$\hcP_{\N}$} (m-2-4)
(m-2-2) edge node[above] {$T^{\isom}_{2, \N^{(1)} }$} (m-2-3)
(m-2-3) edge node[above] {$T^{\isom}_{1, \N}$} (m-2-4)
(m-2-1) edge[bend right = 18] node[below] {$T_{\N}$} (m-2-4);
\end{tikzpicture}
\end{equation}
where $\hat{T} := \hat{T}_1 \circ \hat{T}_2$ and $T_{\N}$ is the corresponding
$\GT$-shadow with the target $\N$. 

Thus 
$$
(\N^{\hat{T}_1})^{\hat{T}_2} = \N^{\hat{T}_1 \circ \hat{T}_2}.
$$
We proved that formula \eqref{That-acts} defines a (right) action of $\GTh$ on 
$\NFI_{\PB_4}(\B_4)$. 

Let us denote by $\GTh_{\NFI}$ the corresponding transformation groupoid. The set of objects of 
$\GTh_{\NFI}$ is $\NFI_{\PB_4}(\B_4)$ and, for $\N^{(1)}, \N^{(2)} \in \NFI_{\PB_4}(\B_4)$,
the set of morphisms $\GTh_{\NFI}(\N^{(1)},\N^{(2)})$ consists of elements $\hat{T} \in \GTh$ such 
that $(\N^{(2)})^{\hat{T}} = \N^{(1)}$.

The diagram in \eqref{diag-T2-T1-N} tells us that, if $\hat{T}:= \hat{T}_1 \circ \hat{T}_2$, then 
$$
T^{\isom}_{\N} = T^{\isom}_{1, \N} \circ T^{\isom}_{2, \N^{(1)}}\,.
$$
Thus the assignment $\hat{T} \mapsto T_{\N}$ indeed
defines a functor $\PR$ from $\GTh_{\NFI}$ to $\GTSh$.
(On the level of objects, the functor $\PR$ operates as the identity map.)
\end{proof}

\bigskip

Recall \cite[Section 2.6]{GTshadows} that a $\GT$-shadow $[m,f] \in \GT(\N)$ is called \e{genuine} if
there exists $\hat{T} \in \GTh$ such that $T_{m,f} = T_{\N}$, i.e. $[m,f]$ comes from 
an element of $\GTh$. If such $\hat{T} \in \GTh$ does not exist then the $\GT$-shadow 
$[m,f]$ is called \e{fake}.

It was shown in \cite[Section 2]{GTshadows} that genuine $\GT$-shadows satisfy additional conditions. 

A $\GT$-shadow is called \e{practical}, if it can be represented by a pair $(m,f)$ 
where\footnote{Recall that $\F_2$ is identified with the subgroup $\lan x_{12}, x_{23} \ran \le \PB_3$.} 
$f \in \F_2$. Since every genuine $\GT$-shadow is practical\footnote{See, for example,  
\cite[Proposition 2.20]{GTshadows}.}, in this paper, we assume that all $\GT$-shadows are practical. 
In particular, $\GT(\N)$ denotes the set of all practical $\GT$-shadows with the 
target $\N$. Furthermore, we will use the same notation $\GTSh$ for the (sub)groupoid 
of a practical $\GT$-shadows. 

If $[m_1,f_1] \in \GTSh(\N^{(2)}, \N^{(1)})$, $[m_2,f_2] \in \GTSh(\N^{(3)}, \N^{(2)})$ and 
\begin{equation}
\label{compose}
\begin{array}{c}
m : = 2 m_1 m_2 + m_1 + m_2\,, \\[0.3cm]
f(x,y) : = f_1(x,y)\, f_2(x^{2 m_1+1}, f_1(x,y)^{-1} y^{2 m_1+1} f_1(x,y)),
\end{array}
\end{equation}
then the pair $(m,f)$ represents a $\GT$-shadow in $\GTSh(\N^{(3)}, \N^{(1)})$
and $[m,f] := [m_1,f_1] \circ [m_2,f_2]$, i.e. formula \eqref{compose} defines the 
composition of (practical) $\GT$-shadows. For more details, see \cite[Remark 2.15]{GTshadows}.

It was proved in \cite[Proposition 2.20]{GTshadows} that, for every genuine 
$\GT$-shadow $[m,f] \in \GT(\N)$, the coset $f\N_{\F_2}$ 
belongs to 
the commutator subgroup 
$[\F_2/\N_{\F_2}, \F_2/\N_{\F_2}]$
\begin{equation}
\label{pty-charming}
f \N_{\F_2} \in [\F_2/\N_{\F_2}, \F_2/\N_{\F_2}].
\end{equation}
$\GT$-shadows satisfying this additional property are called \e{charming}.

Just as in \cite{GTshadows}, the notation $\GT^{\hs}(\N)$ is reserved for the subset 
of charming $\GT$-shadows in $\GT(\N)$. Due to \cite[Proposition 2.22]{GTshadows}, 
charming $\GT$-shadows form a subgroupoid $\GTSh^{\hs}$ of $\GTSh$. 
\begin{remark}  
\label{rem:onto-is-redundant}
Due to Proposition \ref{prop:T-F-2-onto}, the condition about the homomorphism 
$T^{\F_2}_{m,f}$ in \cite[Definition 2.19]{GTshadows} is redundant. In other words, 
a $\GT$-shadow $[m,f] \in \GT(\N)$ is charming if and only if condition \eqref{pty-charming}
is satisfied.
\end{remark}
\begin{remark}  
\label{rem:to-GTSh-hs}
Since, for every $\hat{T} \in \GTh$ and $\N \in \NFI_{\PB_4}(\B_4)$, the $\GT$-shadow
$T_{\N} \in \GT(\N)$ is charming, the functor $\PR : \GTh_{\NFI} \to \GTSh$ 
from Proposition \ref{prop:GThat-NFI} lands in the subgroupoid $\GTSh^{\hs}$. 
\end{remark}

\bigskip
Recall that the groupoid $\GTSh^{\hs}$ is highly disconnected. Indeed, if $\K, \N \in \NFI_{\PB_4}(\B_4)$
have different indices in $\PB_4$, then $\GTSh^{\hs}(\K, \N)$ is empty. Just as in \cite{GTshadows}, 
$\GTSh^{\hs}_{\conn}(\N)$ denotes the connected component of $\N$ in the groupoid $\GTSh^{\hs}$. 
Elements $\N \in \NFI_{\PB_4}(\B_4)$ for which the connected component 
$\GTSh^{\hs}_{\conn}(\N)$ has exactly one object are called \e{isolated} and they play a special role:

\begin{itemize}

\item For every isolated element $\N \in \NFI_{\PB_4}(\B_4)$, $\GT^{\hs}(\N)$ is a finite group. 

\item Due to \cite[Proposition 3.3]{GTshadows}, the subposet  $\NFI^{isolated}_{\PB_4}(\B_4)$ of isolated 
elements in $\NFI_{\PB_4}(\B_4)$ is coinitial: for every $\N \in \NFI_{\PB_4}(\B_4)$, there exists 
$\K \in \NFI^{isolated}_{\PB_4}(\B_4)$ such that $\K \le \N$.  

\item Due to \cite[Proposition 3.6]{GTshadows}, $\NFI^{isolated}_{\PB_4}(\B_4)$ is closed under 
taking finite intersections. 

\item Due to \cite[Proposition 3.7]{GTshadows}, the assignment $\N \to \GT^{\hs}(\N)$ upgrades 
to a functor $\ML$ from the poset  $\NFI^{isolated}_{\PB_4}(\B_4)$ to the category of finite groups. 

\item Finally, due to \cite[Theorem 3.8]{GTshadows}, the limit of $\ML$ is isomorphic to $\GTh$. 

\end{itemize}

\subsection{Permutation pairs, permutation triples and child's drawings}
\label{sec:dessins}

A \e{permutation pair} of degree $d$ is a pair $c = (c_1,c_2) \in S_d \times S_d$ for which the 
subgroup $\lan c_1, c_2 \ran \le S_d$ is transitive. We consider permutation pairs with the 
action of $S_d$ by conjugation
$$
h(c) := (h c_1 h^{-1}, h c_2 h^{-1}), \qquad h \in S_d.
$$ 

Recall that a \e{child's drawing} of degree $d$ is a conjugacy class $[c]$ of a 
permutation pair $c = (c_1, c_2)$. Let us denote by $\ct$ the standard map 
from $S_d$ to the set $\fP_d$ of partitions of $d$: for $h \in S_d$, $\ct(h)$ is the 
cycle structure of $h$.

The (conjugacy class of the) permutation group  $\lan c_1, c_2 \ran \le S_d$ is called 
the \e{monodromy group} of the child's drawing $[c]$. 

In the literature\footnote{This list of references is far from complete.}  \cite{AADKKLNS}, 
\cite{dynamical},  \cite{Jordan}, \cite{Girondo}, \cite{Goins}, \cite{Sketch}, \cite{Lando-Z}, 
\cite{Leila-dessins}, \cite{Voight-Belyi}, \cite{Zapponi},
child's drawings (of degree $d$) are often represented by permutation triples, i.e. 
elements $(g_1, g_2, g_3) $ in $(S_d)^3$ such that 

\begin{itemize}

\item $g_1 g_2 g_3 = \id$ and
 
\item the subgroup $\lan g_1, g_2, g_3 \ran \le S_d$ is transitive.

\end{itemize}
The assignment
\begin{equation}
\label{pairs-to-triples}
(c_1, c_2) \mapsto (c_1, c_2, c_2^{-1} c_1^{-1})
\end{equation}
gives us an obvious bijection from the set of permutation pairs (of degree $d$) to the set 
of permutation triples (of degree $d$). $S_d$ acts on permutation triples by conjugation
$$
h (g_1, g_2, g_3) := (h g_1 h^{-1},  h g_2 h^{-1}, h g_3 h^{-1})
$$
and this action is compatible with bijection \eqref{pairs-to-triples}.

The triple of partitions $(\ct(c_1), \ct(c_2), \ct(c_2^{-1} c_1^{-1}))$ is called the \e{passport} of 
a child's drawing $[c]$.

\subsection{Representing child's drawings by group homomorphisms}
\label{sec:via-homomorph}

It is convenient to represent a child's drawing $[c]$ of degree $d$ by the group homomorphism 
$\psi: \F_2 \to S_d$: 
$$
\psi(x):= c_1, \qquad \psi(y):= c_2.
$$
So we denote by 
\begin{equation}
\label{Hom-tran-Sd}
\Hom_{\tran}(\F_2, S_d)
\end{equation}
the set of group homomorphisms $\psi: \F_2 \to S_d$ for which $\psi(\F_2)$ is transitive. 

Two homomorphisms $\psi, \ti{\psi} \in \Hom_{\tran}(\F_2, S_d)$ represent the 
same child's drawing if and only if $\exists~ h \in S_d$ such that 
$$
\ti{\psi}(w) = h \psi(w) h^{-1}, \qquad \forall~ w \in \F_2.
$$
In other words, the set of child's drawing of degree $d$ can be identified with the set 
of orbits 
\begin{equation}
\label{Hom-tran-Sd-orbits}
\Hom_{\tran}(\F_2, S_d)_{S_d}
\end{equation}
of the $S_d$-action on $\Hom_{\tran}(\F_2, S_d)$ by conjugation.  

Note that $\ker(\psi)$ depends only on the child's drawing $[\psi]$
but not on a particular choice of a representative $\psi \in \Hom_{\tran}(\F_2, S_d)$.  

\bigskip

Child's drawings can also be represented by continuous group homomorphisms from $\Fh_2$ to $S_d$. 
The goal of the following proposition is to recall this equivalent description:
\begin{prop}  
\label{prop:no-hat-yes-hat}
Let $\Hom(\Fh_2, S_d)$ be the set of continuous group 
homomorphisms\footnote{Recall that every finite group (for example, $S_d$) is considered 
with the discrete topology.}
$\Fh_2 \to S_d$ and 
$$
\Hom(\Fh_2, S_d)_{\tran} := \big\{ \hat{\psi} \in \Hom(\Fh_2, S_d) ~|~ 
\hat{\psi}(\Fh_2) \txt{ is a transitive subgroup of } S_d \big\}.
$$
The assignment 
\begin{equation}
\label{restrict-F2}
\hat{\psi} \mapsto \hat{\psi}\big|_{\F_2}
\end{equation}
gives us a bijection from $\Hom(\Fh_2, S_d)$ (resp. from $\Hom_{\tran}(\Fh_2, S_d)$)
to $\Hom(\F_2, S_d)$ (resp. to $\Hom_{\tran}(\F_2, S_d)$). 
This assignment is compatible with the action of $S_d$ by conjugation 
and it gives us a bijection between the set 
\begin{equation}
\label{Fhat-Sd-orbits}
\Hom_{\tran}(\Fh_2, S_d)_{S_d}
\end{equation}
of orbits of the $S_d$-action and the set of child's drawings of degree $d$.  
\end{prop}  
\begin{proof} We will prove this proposition by constructing the inverse of the assignment 
in \eqref{restrict-F2}. 

Let $\psi \in \Hom(\F_2, S_d)$ and $\K : = \ker(\psi)$. The homomorphism 
$\psi: \F_2 \to S_d$ factors as follows
$$
\psi = \psi_{\K} \circ \cP_{\K},
$$
where the homomorphism $\psi_{\K} : \F_2/\K \to S_d$ is defined by the formula 
\begin{equation}
\label{psi-K}
\psi_{\K}(w \K) := \psi(w).
\end{equation}

We denote by $\hat{\psi}$ the continuous homomorphism $\Fh_2 \to S_d$ defined 
by the formula 
\begin{equation}
\label{psi-hat}
\hat{\psi} := \psi_{\K} \circ \hcP_{\K}\,.
\end{equation}

We claim that the assignment 
$$
\psi \mapsto \hat{\psi}
$$
gives us the inverse of the map from $\Hom(\Fh_2, S_d)$ to  $\Hom(\F_2, S_d)$
defined in \eqref{restrict-F2}.

Indeed, $\hat{\psi} \big|_{\F_2}$ clearly coincides with $\psi$. 
Thus it remains to show that, if 
\begin{equation}
\label{psi-from-vf-hat}
\psi = \hat{\vf}\big|_{\F_2}  
\end{equation}
for $\hat{\vf} \in \Hom(\Fh_2, S_d)$, then $\hat{\psi}$ coincides with $\hat{\vf}$. 

Let $K:= \ker(\psi)$ and $\psi_{\K}$ be the homomorphism $\F_2/\K \to S_d$ defined by 
\eqref{psi-K}, where $\psi$ is defined in \eqref{psi-from-vf-hat}. As above, we set 
$\hat{\psi}:= \psi_{\K} \circ \hcP_{\K}$ and observe that 
$$
\hat{\psi}\big|_{\F_2} = \hat{\vf}\big|_{\F_2}.  
$$
Since $\F_2$ is dense in $\Fh_2$, $S_d$ is Hausdorff and $\hat{\psi}$, $\hat{\vf}$ are continuous, we conclude 
that $\hat{\psi} = \hat{\vf}$. 

It is easy to see that, for every $\hat{\psi} \in \Hom(\Fh_2, S_d)$, $\psi(\F_2) = \hat{\psi}(\Fh_2)$. 
Moreover the assignments \eqref{restrict-F2} and $\psi \mapsto \hat{\psi}$ are compatible with 
the (adjoint) action of $S_d$. Thus the remaining statements of the proposition are obvious. 
\end{proof}
\begin{remark}  
\label{rem:extend-by-continuity}
In the above proof, we showed that every group homomorphism $\psi: \F_2 \to S_d$
extends uniquely to a continuous group homomorphism $\hat{\psi}: \Fh_2 \to S_d$. 
Equation \eqref{psi-hat} is an explicit definition of this extension. 
See also \cite[Lemma 1.1.16, (b)]{RZ-profinite}.
\end{remark}
\begin{remark}  
\label{rem:represent-by-psi-hat}
Depending on the context, we will represent a child's drawing of degree $d$ by a group homomorphism 
$\psi: \F_2 \to S_d$ or by a continuous group homomorphism $\hat{\psi}: \Fh_2 \to S_d$.  
Of course, due to Proposition \ref{prop:no-hat-yes-hat}, we have 
$$
[\hat{\psi}] = \big[ \hat{\psi} \big|_{\F_2} \big]
$$
for every $\hat{\psi} \in \Hom_{\tran}(\Fh_2, S_d)$.
\end{remark}
\begin{remark}  
\label{rem:conj-finite-index}
There is a natural bijection between the set $\Hom_{\tran}(\F_2, S_d)_{S_d}$ and 
the set of conjugacy classes of index $d$ subgroups $H$ in $\F_2$. This bijection sends
$[\psi] \in \Hom_{\tran}(\F_2, S_d)_{S_d}$ to the conjugacy class (in $\F_2$) of the 
subgroup\footnote{Since the subgroup $\psi(\F_2) \le S_d$ acts transitively on $\{1,2,\dots, d\}$,
the subgroup 
$H_{\psi, i} := \{h \in \F_2 ~|~ \psi(h)(i) = i\} \le \F_2$ is conjugate to $H_{\psi}$ in $\F_2$ for every 
$i \in \{1,2, \dots, d\}$.}
$$
H_{\psi} := \{h \in \F_2 ~|~ \psi(h)(1) = 1\}.
$$
The inverse of this correspondence operates as follows: given a subgroup $H \le \F_2$ of index 
$d$, we choose a bijection between the set $\F_2/H$ of left cosets of $H$ and the set 
$\{1,2, \dots, d\}$; then the canonical action of $\F_2$ on $\F_2/H$ gives us 
a desired group homomorphism $\psi : \F_2 \to S_d$. Thus every child's drawing of degree $d$
can be represented by an index $d$ subgroup in $\F_2$. 
\end{remark}
\begin{remark}  
\label{rem:coverings-Belyi-pairs}
Recall \cite[Section 1.3]{Hatcher} that conjugacy classes of index $d$ subgroups in $\F_2$ 
are in bijection with connected degree $d$ covering spaces of  $\CP^1 \setminus \{0,1,\infty\}$. 
We say that a child's drawing $[H]$ represented by a finite index subgroup $H \le \F_2$ is \e{Galois}
if the corresponding covering of $\CP^1 \setminus \{0,1,\infty\}$ is Galois. Due to \cite[Proposition 1.39]{Hatcher}, 
the child's drawing $[H]$ is Galois if and only if $H$ is a normal subgroup of $\F_2$.  

Recall that a \e{Belyi pair} is a pair $(X,\ga)$, where $X$ is a smooth projective curve 
defined over $\ol{\bbQ}$ and $\ga : X \to \bbP^1_{\ol{\bbQ}}$ is a morphism of curves
unramified outside of the set $\{0,1,\infty\}$. 
Using \cite[Proposition 4.5.13]{Szamuely} and \cite[Theorem 4.6.10]{Szamuely}, one can 
show\footnote{Some mathematicians also like to cite \cite{SGA1}.} that isomorphism classes 
of connected degree $d$ coverings of  $\CP^1 \setminus \{0,1,\infty\}$ are in bijection 
with isomorphism classes of degree $d$ Belyi pairs. Thus child's drawings 
can be also represented by Belyi pairs. The standard action of $G_{\bbQ}$ on 
child's drawings is often defined in terms of Belyi pairs.
  
For more details, we refer to the reader to books \cite{Girondo}, \cite{Lando-Z},  
blog \cite{Goins}, and survey \cite{Voight-Belyi}.
The curious reader may also try to ``surf through'' the database 
of Belyi pairs \cite{Belyi_database}. 
\end{remark}

\subsection{The action of $\GTh$ on child's drawings and the functor $\cA: \GTh_{\NFI} \to \Dessin$}
\label{sec:GTh-action}

The goal of the following proposition is to recall the action of $\GTh$ on child's drawings. 
\begin{prop}  
\label{prop:GTh-acts}
Let $\hat{T} \in \GTh$ and  $\hat{T}_{\Fh_2}$ be the corresponding continuous automorphism 
of $\Fh_2 \le \wh{\PB}_3 = \PaB\big( (12)3,1(23) \big)$. 
Then the formula
\begin{equation}
\label{GTh-acts}
\hat{\psi}^{\hat{T}} := \hat{\psi} \circ \hat{T}_{\Fh_2}
\end{equation}
defines a right action of $\GTh$ on the set $\Hom(\Fh_2, S_d)$. 
This action descends to the action of $\GTh$ on $\Hom_{\tran}(\Fh_2, S_d)$ and on 
$\Hom_{\tran}(\Fh_2, S_d)_{S_d}$. 
\end{prop}  
\begin{proof}
It is obvious that, for $\hat{T}^{(1)}, \hat{T}^{(2)} \in \GTh$, 
$$
\big(\hat{T}^{(1)} \circ \hat{T}^{(2)}\big)_{\Fh_2} = \hat{T}^{(1)}_{\Fh_2} \circ \hat{T}^{(2)}_{\Fh_2}\,.
$$
Thus the first statement of the proposition is obvious. 

Since $\hat{T}_{\Fh_2}$ is an isomorphism $\Fh_2 \iso \Fh_2$, the subgroups  
$\hat{\psi} \circ \hat{T}_{\Fh_2}(\Fh_2)\le S_d$ and $\hat{\psi}(\Fh_2) \le S_d$ coincide.  
Moreover, the resulting action of $\GTh$ on $\Hom_{\tran}(\Fh_2, S_d)$ clearly commutes 
with the action of $S_d$. Hence the remaining two statements of the proposition are 
also obvious. 
\end{proof}

\bigskip
\noindent
{\bf On the Ihara embedding.} Recall \cite{Ihara} that $G_{\bbQ}$ injects into the group $\GTh$
\begin{equation}
\label{Ihara-GQ-GTh}
G_{\bbQ}  \hookrightarrow \GTh,
\end{equation}
and the standard action of $G_{\bbQ}$ on child's drawings agrees with 
homomorphism \eqref{Ihara-GQ-GTh} and the above action of $\GTh$ on child's drawings. 
For details, please see \cite[Proposition 1.6]{Ihara}, \cite[Theorem 1.7]{Ihara} and 
\cite[Section 3.2]{Leila-survey}. We will call homomorphism \eqref{Ihara-GQ-GTh} the \e{Ihara embedding}.

\bigskip

For our purposes, it is convenient to describe the action of $\GTh$ on child's drawing 
in terms of a functor from the transformation groupoid $\GTh_{\NFI}$ to certain category 
assembled from child's drawings and elements of the poset $\NFI_{\PB_4}(\B_4)$. 

\begin{defi}  
\label{dfn:subordinate}
Let $\N \in \NFI_{\PB_4}(\B_4)$ and $[\psi]$ be a child's drawing of degree $d$
represented by a homomorphism $\psi: \F_2 \to S_d$. 
We say that a child's drawing $[\psi]$ is \e{subordinate} to $\N$ if 
\begin{equation}
\label{subord-cond}
\N_{\F_2} \le \ker(\psi).
\end{equation}
If $[\psi]$ is subordinate to $\N$, then we say that $\N$ \e{dominates} $[\psi]$.  
\end{defi}   
It is easy to see that condition \eqref{subord-cond} does not depend on the 
choice of representing homomorphism $\psi$. Moreover, if a child's drawing is represented 
by a continuous homomorphism $\hat{\psi}: \Fh_2 \to S_d$ then $[\hat{\psi}]$ is subordinate 
to $\N$ if and only if $\N_{\F_2} \le  \ker(\hat{\psi}|_{\F_2})$. 
We denote by $\Dessin(\N)$ the set of child's drawings subordinate to $\N$. 

It is clear that, if $\K \le \N$
($\K,\N \in  \NFI_{\PB_4}(\B_4)$) and child's drawing $[\psi]$ is subordinate to $\N$ then 
$[\psi]$ is also subordinate to $\K$. In other words, if $\K \le \N$, then 
$\Dessin(\N) \subset \Dessin(\K)$.

Let us introduce the category $\Dessin$ whose objects are elements of $\NFI_{\PB_4}(\B_4)$. 
For $\N^{(1)}, \N^{(2)} \in \NFI_{\PB_4}(\B_4)$, morphisms from $\N^{(1)}$ to $\N^{(2)}$ are 
functions from $\Dessin(\N^{(1)})$ to $\Dessin(\N^{(2)})$.  

\begin{prop}  
\label{prop:GTh-NFI-to-Dessin}
For every $\hat{T} \in \GTh$ and $[\hat{\psi}] \in \Dessin(\N)$, 
the child's drawing represented by 
the homomorphism 
$$
\hat{\psi} \circ \hat{T}\big|_{\F_2} : \F_2 \to S_d
$$
is subordinate to $\N^{\hat{T}}$. 
Moreover, the formulas
\begin{equation}
\label{A-functor}
\A(\N):=\N, 
\qquad 
\A(\hat{T})([\hat{\psi}]) := [\hat{\psi} \circ \hat{T}]
\end{equation}
define a cofunctor
$$
\A : \GTh_{\NFI} \to \Dessin.
$$ 
\end{prop}  
\begin{proof} Identifying $\F_2$ and $\wh{\F}_2$ with the corresponding 
subgroups of $\wh{\PB}_3 = \PaB\big( (1,2)3, (1,2)3 \big)$ and using 
\eqref{T-N} we see that $\N^{\hat{T}}_{\F_2}$ is the kernel of the homomorphism 
$$
\hcP_{\N_{\F_2}} \circ \hat{T}\big|_{\F_{2}} : \F_2 \to \F_2/\N_{\F_2}\,.
$$ 

Let $\psi: = \hat{\psi}|_{\F_2}$. 
Since $\N_{\F_2} \le \ker(\psi)$, $\psi$ induces the group homomorphism 
$$
\ti{\psi} :  \F_2/\N_{\F_2} \to S_d,
\qquad 
\ti{\psi}(w \N_{\F_2}) := \psi(w).
$$

Composing $\ti{\psi}$ with $\hcP_{\N_{\F_2}}$ we get a continuous 
group homomorphism $\ti{\psi} \circ \hcP_{\N_{\F_2}} : \Fh_2 \to S_d$. 
Since $\ti{\psi} \circ \hcP_{\N_{\F_2}} \big|_{\F_2} = \hat{\psi}\big|_{\F_2}$, 
$\F_2$ is dense in $\Fh_2$ and $S_d$ is Hausdorff, we conclude that 
$$
\hat{\psi} = \ti{\psi} \circ \hcP_{\N_{\F_2}}\,.
$$

Hence
\begin{equation}
\label{key4N-T-F2}
\hat{\psi} \circ  \hat{T}\big|_{\F_{2}} =   \ti{\psi} \circ (\hcP_{\N_{\F_2}} \circ \hat{T}\big|_{\F_{2}}).
\end{equation}
Since the kernel of $\hcP_{\N_{\F_2}} \circ \hat{T}\big|_{\F_{2}}$ is $\N^{\hat{T}}_{\F_2}$, 
\eqref{key4N-T-F2} implies that 
$$
\N^{\hat{T}}_{\F_2} ~\le~ \ker\big( \hat{\psi} \circ  \hat{T}\big|_{\F_{2}} \big).  
$$ 

The first statement of the proposition is proved. 

Using the statement we just proved and Proposition \ref{prop:GTh-acts} we see 
that, for all $\N^{(1)}, \N^{(2)} \in \NFI_{\PB_4}(\B_4)$, 
the second formula in \eqref{A-functor} defines a map 
$$
\GTh_{\NFI}(\N^{(1)}, \N^{(2)} ) \to \Dessin(\N^{(2)}, \N^{(1)}). 
$$ 

Moreover, since $\hat{\psi} \circ \id_{\Fh_2} = \hat{\psi}$ and 
$$
\hat{\psi}  \circ  (\hat{T}^{(1)} \circ  \hat{T}^{(2)})_{\Fh_2}  = 
(\hat{\psi} \circ \hat{T}^{(1)}_{\Fh_2}) \circ \hat{T}^{(2)}_{\Fh_2}\,, 
\quad 
\forall~~
\hat{T}^{(1)},~ \hat{T}^{(2)} \in \GTh, ~~
\hat{\psi} \in \Hom(\Fh_2, S_d),
$$ 
we conclude that the formulas in \eqref{A-functor} indeed define a cofunctor 
$\A$ from the transformation groupoid $\GTh_{\NFI}$ to the category $\Dessin$. 
\end{proof}

\bigskip

Using the Ihara embedding $G_{\bbQ} \hookrightarrow \GTh $ and the action 
of $\GTh$ on $\NFI_{\PB_4}(\B_4)$, we get the action of $G_{\bbQ}$ on  $\NFI_{\PB_4}(\B_4)$. 
We denote by  $G_{\bbQ, \NFI}$
the corresponding transformation groupoid and by $\Ih$ the natural functor 
\begin{equation}
\label{Ih-functor}
\Ih: G_{\bbQ, \NFI} \to \GTh_{\NFI}
\end{equation}
coming from the Ihara embedding.

Composing $\Ih$ with the cofunctor $\A :  \GTh_{\NFI} \to \Dessin$, we get 
a cofunctor from the groupoid $G_{\bbQ, \NFI}$ to the category $\Dessin$. 
We denote this cofunctor by $\A^{\bbQ}$: 
\begin{equation}
\label{A-bbQ}
\A^{\bbQ} :=  \A \circ \Ih : G_{\bbQ, \NFI} \to \Dessin. 
\end{equation}

\section{The action of $\GT$-shadows on child's drawings}
\label{sec:GTshadows-act}

It is convenient to introduce the action of $\GT$-shadows on child's drawings 
as a cofunctor from the groupoid $\GTSh$ to the category $\Dessin$:
\begin{thm}  
\label{thm:Action}
Let $\N^{(1)}, \N^{(2)} \in \NFI_{\PB_4}(\B_4)$, $[m,f] \in \GTSh(\N^{(2)},\N^{(1)})$ and
$[\psi] \in  \Dessin(\N^{(1)})$
be a child's drawing represented by a homomorphism $\psi : \F_2 \to S_d$. 
Let $\psi^{(m,f)}: \F_2 \to S_d$ be the homomorphism defined by the formulas
\begin{equation}
\label{def-the-action}
\psi^{(m,f)}(x) : = \psi(x^{2m+1}), \qquad 
\psi^{(m,f)}(y) : = \psi(f^{-1} y^{2m+1} f).
\end{equation}
Then 
\begin{itemize}

\item $\psi^{(m,f)}$ does not depend on the choice of the pair $(m,f)$ representing the $\GT$-shadow $[m,f]$;

\item $\psi^{(m,f)}$ represents a child's drawing of degree $d$ subordinate to $\N^{(2)}$. 

\end{itemize}
Moreover, the formulas 
\begin{equation}
\label{A-sh-cofunctor}
\A^{sh}(\N): = \N, \qquad 
\A^{sh}([m,f])([\psi]) := [\psi^{(m,f)}]
\end{equation}
define a cofunctor $\A^{sh} : \GTSh \to \Dessin$.   
\end{thm}  
\begin{proof} Since $\N^{(1)}_{\F_2} \le \ker(\psi)$ the formula 
$$
\ti{\psi}(w \N^{(1)}_{\F_2}) : = \psi(w)
$$
defines a group homomorphism $\ti{\psi} : \F_2 /\N^{(1)}_{\F_2} \to S_d$ and 
$\ti{\psi}\big(\F_2 /\N^{(1)}_{\F_2} \big) = \psi(\F_2)$. In particular, the subgroup 
$\ti{\psi}\big(\F_2 /\N^{(1)}_{\F_2} \big)$ is transitive. 

Due to \eqref{T-F-2-def}, 
\begin{equation}
\label{psi-m-f}
\psi^{(m,f)} = \ti{\psi} \circ T^{\F_2}_{m,f}\,.
\end{equation}
Since the homomorphism $T^{\F_2}_{m,f}: \F_2 \to  \F_2 /\N^{(1)}_{\F_2}$ does not 
depend on the choice of the representative $(m,f)$ of the $\GT$-shadow $[m,f]$, 
the first statement of the proposition is proved. 

Since $T^{\F_2}_{m,f}: \F_2 \to  \F_2 /\N^{(1)}_{\F_2}$ is onto 
$\psi^{(m,f)} (\F_2) = \ti{\psi}\big(\F_2 /\N^{(1)}_{\F_2} \big)$. Hence
$\psi^{(m,f)} (\F_2)$ is a transitive subgroup of $S_d$.

We know that $\ker(T^{\F_2}_{m,f}) = \N^{(2)}_{\F_2}$. Hence \eqref{psi-m-f}
implies that 
$$
\N^{(2)}_{\F_2} \le \ker(\psi^{(m,f)}).
$$
We proved the second statement of the proposition. 

It is easy to see that $[\psi^{(m,f)}]$ does not depend on the representative 
$\psi$ of the child's drawing $[\psi]$. Thus the formula 
$$
\A^{sh}([m,f])([\psi]) := [\psi^{(m,f)}]
$$
defines a map $\GTSh(\N^{(2)}, \N^{(1)}) \to \Dessin(\N^{(1)}, \N^{(2)})$.

It is also easy to see that, for every $[\psi] \in \Dessin(\N^{(1)})$,  
$[\psi^{(0,1_{\F_2})}] = [\psi]$.

Thus it remains to show that, for all $[m_1, f_1] \in \GTSh(\N^{(2)}, \N^{(1)})$,  $[m_2, f_2] \in \GT(\N^{(2)})$,
and $[\psi] \in \Dessin(\N^{(1)})$, 
\begin{equation}
\label{A-sh-composition-ok}
\A^{sh}([m_1,f_1] \circ [m_2, f_2])([\psi]) = \big[ (\psi^{(m_1,f_1)})^{(m_2, f_2)} \big].
\end{equation}

Equation \eqref{A-sh-composition-ok} can be verified directly using \eqref{compose}. 

Here is another way to prove \eqref{A-sh-composition-ok}. 
Let $[m,f] := [m_1,f_1] \circ [m_2, f_2]$. Since the composition of $\GT$-shadows
is defined via the identification 
of elements of $\GT(\N)$ with isomorphisms \eqref{T-m-f-isom}, we have
$$
T_{m,f} =  T^{\isom}_{m_1, f_1} \circ T_{m_2, f_2}\,.
$$ 
Hence  
\begin{equation}
\label{T-F-2-composition-ok}
T^{\F_2}_{m,f} =  T^{\F_2, \isom}_{m_1, f_1} \circ T^{\F_2}_{m_2, f_2}\,.
\end{equation}

Since $\psi^{(m,f)} = \ti{\psi} \circ T^{\F_2}_{m,f}$, identity \eqref{T-F-2-composition-ok} 
implies that 
\begin{equation}
\label{A-sh-compose}
\psi^{(m,f)}  = (\ti{\psi} \circ T^{\F_2, \isom}_{m_1, f_1}) \circ T^{\F_2}_{m_2, f_2}\,.
\end{equation}

Since $\ti{\psi} \circ T^{\F_2, \isom}_{m_1, f_1} \circ \cP_{\N^{(2)}_{\F_2}} = \psi^{(m_1, f_1)}$, 
identity \eqref{A-sh-compose} implies the desired equation in \eqref{A-sh-composition-ok}. 
\end{proof}

\bigskip

The action of $\GT$-shadows on child's drawings is compatible with the action of $\GTh$ in the following sense:
\begin{thm}  
\label{thm:compatible}
Let $\PR$ be the natural functor from $\GTh_{\NFI}$ to $\GTSh^{\hs}$ introduced 
in the proof of Proposition \ref{prop:GThat-NFI} and let $\A$ be the cofunctor 
from $\GTh_{\NFI}$ to $\Dessin$ defined in Proposition \ref{prop:GTh-NFI-to-Dessin}. 
The diagram of (co)functors 
\begin{equation}
\label{A-A-sh-diag}
\begin{tikzpicture}
\matrix (m) [matrix of math nodes, row sep=1.8em, column sep=1.8em]
{\GTh_{\NFI} & ~ & \GTSh^{\hs} \\
 ~ & \Dessin & ~  \\};
\path[->, font=\scriptsize]
(m-1-1) edge node[above] {$\PR$} (m-1-3)
 edge node[left] {$\A~~$} (m-2-2)
(m-1-3) edge node[right] {$\A^{sh}$} (m-2-2);  
\end{tikzpicture}
\end{equation}
commutes ``on the nose''. 
\end{thm}  
\begin{proof}
Since all three functors operate as the identity map on the level of objects, we need to show that, 
for every $\N \in \NFI_{\PB_4}(\B_4)$, $\hat{T} \in \GTh$ and $[\psi] \in \Dessin(\N)$ we have 
\begin{equation}
\label{actions-ok}
\hat{\psi} \circ \hat{T}\big|_{\F_2} = \ti{\psi} \circ T^{\F_2}_{\N}\,, 
\end{equation}
where $\hat{\psi} : \Fh_2 \to S_d$ is the continuous group homomorphism that extends 
$\psi: \F_2 \to S_d$, and $\ti{\psi}$ is the group homomorphism $\F_2 / \N_{\F_2} \to S_d$ defined 
by the formula 
$$
\ti{\psi} \big(w \, \N_{\F_2} \big) := \psi(w).  
$$
 
Due to the relation defining the $\GT$-shadow
$T_N :  \PaB^{\le 4} \to \PaB^{\le 4}/ \sim_{\N}$ in terms 
of $\hat{T} \in \GTh$ (see equation \eqref{T-N}), 
the diagram 
\begin{equation}
\label{diag-F-2-hat}
\begin{tikzpicture}
\matrix (m) [matrix of math nodes, row sep=2.3em, column sep=2.3em]
{\wh{\F}_2 & \wh{\F}_2 \\
\F_2  &  \F_2/\N_{\F_2}  \\};
\path[->, font=\scriptsize]
(m-1-1) edge node[above] {$\hat{T}|_{\wh{\F}_2}$} (m-1-2)  
(m-2-1) edge (m-1-1)  edge node[above] {$T_{\N}^{\F_2}$} (m-2-2)
(m-1-2) edge node[left] {$~~\hcP_{\N_{\F_2}}$} (m-2-2);
\end{tikzpicture}
\end{equation}
commutes.

Using equation \eqref{psi-hat} from the proof of Proposition \ref{prop:no-hat-yes-hat} and 
the relation $\N_{\F_2} \le \ker(\psi)$, it is easy to see that 
the diagram 
\begin{equation}
\label{diag-psi-hat-psi-tilde}
\begin{tikzpicture}
\matrix (m) [matrix of math nodes, row sep=2.3em, column sep=2.6em]
{\wh{\F}_2  & ~~ \\
\F_2/\N_{\F_2}   & S_d  \\};
\path[->, font=\scriptsize]
(m-1-1) edge node[left] {$~~\hcP_{\N_{\F_2}}$} (m-2-1) 
edge node[above] {$\hat{\psi}$} (m-2-2)
(m-2-1)   edge node[above] {$\ti{\psi}~~$} (m-2-2);
\end{tikzpicture}
\end{equation}
commutes.

Putting \eqref{diag-F-2-hat} and \eqref{diag-psi-hat-psi-tilde} together, we 
get the following commutative diagram 
$$
\begin{tikzpicture}
\matrix (m) [matrix of math nodes, row sep=2.3em, column sep=2.6em]
{\wh{\F}_2 & \wh{\F}_2  & ~~~\\
\F_2  &  \F_2/\N_{\F_2}  & S_d \\};
\path[->, font=\scriptsize]
(m-1-1) edge node[above] {$\hat{T}|_{\wh{\F}_2}$} (m-1-2)  
(m-2-1) edge (m-1-1)  edge node[above] {$T_{\N}^{\F_2}$} (m-2-2)
(m-1-2) edge node[left] {$\hcP_{\N_{\F_2}}$} (m-2-2)
(m-1-2) edge node[above] {$~~~\hat{\psi}$} (m-2-3)
(m-2-2) edge node[above] {$\ti{\psi}~~~$} (m-2-3);
\end{tikzpicture}
$$

Thus equation \eqref{actions-ok} is proved and Theorem \ref{thm:compatible} follows.  
\end{proof}

\bigskip 

Let us prove that 
\begin{prop}  
\label{prop:subordinate}
For every child's drawing $\cD$, there exists $\N \in \NFI_{\PB_4}(\B_4)$ that 
dominates $\cD$, i.e. $\cD \in \Dessin(\N)$. In fact, for every child's drawing $\cD$, 
there are infinitely many elements $\K \in \NFI_{\PB_4}(\B_4)$ such that 
$\cD \in \Dessin(\K)$.
\end{prop}  
\begin{proof}
Let $\psi$ be a group homomorphism from $\F_2$ to $S_d$ that represents a child's drawing $\cD$. 

The following formulas 
\begin{equation}
\label{PB4-to-Sd}
\begin{array}{c}
\vf(x_{12}) := \psi(x), \qquad 
\vf(x_{23}) := \psi(y), \qquad 
\vf(x_{13}) := \psi(x^{-1}y^{-1}), \\[0.3cm]
\vf(x_{14}) = \vf(x_{24}) = \vf(x_{34}) := 1_{S_d}
\end{array}
\end{equation}
define a group homomorphism $\vf : \PB_4 \to S_d$. 

Indeed, since $x_{13} = x_{12}^{-1} x_{23}^{-1} c$, the first three equations in \eqref{PB4-to-Sd} define 
a group homomorphism from $\PB_3$ to $S_d$. Since the elements $x_{14}, x_{24}, x_{34}$ generate a 
free subgroup of $\PB_4$ and $\PB_4$ is isomorphic to the semi-direct product 
$\PB_3 \ltimes \lan x_{14}, x_{24}, x_{34} \ran$, the formulas in \eqref{PB4-to-Sd} define 
a group homomorphism $\vf : \PB_4 \to S_d$. 

It is clear that 
\begin{equation}
\label{ker-vf-F2}
\ker(\vf \big|_{\PB_3}) \cap \F_2 = \ker(\psi). 
\end{equation}
Unfortunately, in general, $\ker(\vf)$ is not normal in $\B_4$. 

We denote by $\N$ the normal core of $\ker(\vf)$ in $\B_4$. Since $\ker(\vf)$ has finite index in $\B_4$, 
so does $\N$. In addition, $\N \le \PB_4$. Thus $\N \in \NFI_{\PB_4}(\B_4)$.

Using the definition of $\N_{\PB_3}$ (see equation (2.4) in \cite[Section 2.2]{GTshadows})
and the inclusion $\N \le \ker(\vf)$, we conclude that 
$$
\N_{\PB_3} \le \ker\big( \vf |_{\PB_3} \big).
$$
Combining this observation with identity \eqref{ker-vf-F2}, we conclude that 
$$
\N_{\F_2} \le \ker(\psi).
$$
Thus $\N$ dominates the child's drawing $[\psi]$. 

To prove the second statement, let us show that, for every $\K \in  \NFI_{\PB_4}(\B_4)$
that dominates $[\psi]$, there exists $\ti{\K} \in  \NFI_{\PB_4}(\B_4)$ such that 
\begin{itemize}

\item $\ti{\K}$ dominates $[\psi]$ and 

\item $\ti{\K}$ is properly contained in $\K$. 

\end{itemize}

Since $\K$ is non-trivial, there exists a non-identity element $w \in \K$. 
Since $\B_4$ is residually finite, there exists $\H \in \NFI(\B_4)$ such that 
$w \notin \H$. We set 
$$
\ti{\K} := \H \cap \K. 
$$  
It is clear that $\ti{\K} \in \NFI_{\PB_4}(\B_4)$ and $\ti{\K}$ is properly contained in $\K$. 

Since $\ti{\K} \le \K$, $\ti{\K}$ also dominates $[\psi]$.
\end{proof}
\begin{remark}  
\label{rem:David-Leila}
Although paper \cite{HS} uses a more restrictive assumption on the analogue of 
$\N \in \NFI_{\PB_4}(\B_4)$, a discussion that is very similar to the proof of 
Proposition \ref{prop:subordinate} can be found on page 225 of \cite{HS}. 
\end{remark}
\begin{remark}  
\label{rem:analogy-fields}
Using Belyi pairs, we can construct a natural analogue\footnote{I came up with the 
category $\Dessin_{\bbQ}$ thanks to a suggestion of a diligent referee.} 
$\Dessin_{\bbQ}$ of the category 
$\Dessin$. Objects of $\Dessin_{\bbQ}$ are number fields 
$K \supset \bbQ$; for a number field $K$, we denote by 
$\Dessin_{\bbQ}(K)$ the set of child's drawings that admit a Belyi pair defined over $K$; finally, 
for number fields $E, K$, the set of morphisms from $E$ to $K$ is the set of all maps 
from $\Dessin_{\bbQ}(E)$ to $\Dessin_{\bbQ}(K)$. The natural analogue of the groupoid 
$\GTh_{\NFI}$ is the transformation groupoid $G^{num fields}_{\bbQ}$: objects of $G^{num fields}_{\bbQ}$ 
are number fields and morphisms from $K$ to $E$ are elements of $g\in G_{\bbQ}$ such that $g(K)= E$. Clearly, 
the action of $G_{\bbQ}$ on child's drawings (defined via Belyi pairs) is nothing but 
a cofunctor from $G^{num fields}_{\bbQ}$ to the category $\Dessin_{\bbQ}$. 

In the same vein, the relation of subordinance from Definition \ref{dfn:subordinate} 
is analogous to a natural relationship between child's drawings and finite Galois extensions of $\bbQ$. 
One can say that a child's drawing $\cD$
is \e{subordinate} to a finite Galois extension $E \supset \bbQ$ if $\cD$ can be represented 
by a Belyi pair $(X,\ga)$ defined over an intermediate field of the extension 
$E \supset \bbQ$. The direct analogue of the first statement of
Proposition \ref{prop:subordinate} for this relation 
is obvious because every finite field extension $K \supset \bbQ$ is contained in a 
finite Galois extension of $\bbQ$. The direct analogue of the second statement of
Proposition \ref{prop:subordinate} is also obvious since the field 
extension $\ol{\bbQ} \supset \bbQ$ is infinite. 

There may be further interesting parallels between the pairs of categories
$$
(\GTh_{\NFI},~ \Dessin) \qquad \txt{and} \qquad
(G^{num fields}_{\bbQ},~ \Dessin_{\bbQ})
$$
and it is tempting to explore these parallels. 
\end{remark}

\bigskip

It is relatively easy to see that, if a child's drawing $[\psi]$ is Galois, then so is 
$[\psi]^g$ for every $g \in G_{\bbQ}$. The corresponding version of this statement 
for the action of $\GT$-shadows requires a proof: 
\begin{prop}  
\label{prop:Galois-to-Galois}
Let $\N \in \NFI_{\PB_4}(\B_4)$, $[\psi] \in \Dessin(\N)$ and 
$[m,f] \in \GT(\N)$.  If the child's drawing $[\psi]$ is Galois 
then so is $[\psi]^{[m,f]}$. 
\end{prop}  
\begin{proof}
Let $\psi: \F_2 \to S_d$ be a homomorphism that represents the child's drawing $[\psi]$. 
Note that $[\psi]$ is Galois if and only if the order of the subgroup $\psi(\F_2) \le S_d$ is $d$. 
(This is equivalent to the statement that the stabilizer of $1$ in $\F_2$ 
coincides with its normal core.)

Let us denote by $\ti{\psi}$ the homomorphism from $\F_2/\N_{\F_2} \to S_d$ defined 
by the formula
$$
\ti{\psi}(w \N_{\F_2}) = \psi(w). 
$$

The child's drawing $[\psi]^{[m,f]}$ is represented by the homomorphism 
$$
\ti{\psi} \circ T^{\F_2}_{m,f} : \F_2 \to S_d. 
$$

Since $T^{\F_2}_{m,f}: \F_2 \to \F_2/\N_{\F_2}$ is onto 
(see Proposition \ref{prop:T-F-2-onto}) and $\ti{\psi}\big(\F_2/\N_{\F_2}\big) = \psi(\F_2)$,
\begin{equation}
\label{the-same-mon-group}
\ti{\psi} \circ T^{\F_2}_{m,f} (\F_2) = \psi(\F_2). 
\end{equation}

Thus the order of the subgroup $\ti{\psi} \circ T^{\F_2}_{m,f} (\F_2)\le S_d$ also coincides 
with the degree $d$. Hence the child's drawing  $[\psi]^{[m,f]}$ is Galois. 
\end{proof}

\bigskip

The following proposition shows that there is a large supply of Galois child's drawings 
whose $\GTSh^{\hs}$-orbits are singletons: 
\begin{prop}  
\label{prop:N-isolated-NF2-great}
For every $\N \in \NFI_{\PB_4}(\B_4)$, the Galois child's drawing $\cD_{\N}$ represented 
by $\N_{\F_2}$ is subordinate to $\N$. Moreover, if $\N$ is isolated, then 
the orbit $\GT^{\hs}(\N)(\cD_{\N})$ is a singleton. 
\end{prop}  
\begin{proof}
We set $d:=|\F_2 : \N_{\F_2}|$ and choose a bijection between the set $\F_2 /\N_{\F_2}$ of 
left cosets and $\{1,2,\dots, d\}$. Then the standard (left) action of $\F_2$ on  $\F_2 /\N_{\F_2}$
gives us a group homomorphism
$$
\psi: \F_2 \to S_d
$$
which represents the child's drawing $\cD_{\N}$. 

Since $\N_{\F_2}$ is normal in $\F_2$, $\ker(\psi)= \N_{\F_2}$. 
Thus $\cD_{\N}$ is subordinate to $\N$. 

For the second statement, we assume that $\N$ is isolated. 

Let $\ti{\psi}$ be the homomorphism $\F_2/\N_{\F_2} \to S_d$ defined by 
the formula 
$$
\ti{\psi}(w \N_{\F_2}) : = \psi(w)
$$
and let $[m,f] \in \GT^{\hs}(\N)$.

The (new) child's drawing $\cD_{\N}^{[m,f]}$ is represented by 
the homomorphism
$$
\psi^{(m,f)} := \ti{\psi} \circ T^{\F_2}_{m,f} : \F_2 \to S_d.
$$
Since $\ker(\psi)= \N_{\F_2}$, the homomorphism $\ti{\psi}$ is injective. 

Since $\N$ is isolated, the kernel of the morphism 
$T_{m,f} : \PaB^{\le 4} \to  \PaB^{\le 4}/\sim_{\N}$ is the compatible equivalence 
relation corresponding to $\N$. 
Hence $\ker(T^{\PB_3}_{m,f}) = \N_{\PB_3}$ and therefore 
$\ker(T^{\F_2}_{m,f}) = \N_{\F_2}$. 

Combining this observation with the injectivity of $\ti{\psi}$, we conclude that 
$$
\ker(\psi^{(m,f)}) = \N_{F_2}.
$$

Due to Proposition \ref{prop:Galois-to-Galois}, the child's drawing $[\psi^{(m,f)}]$ is Galois. 

To show that $[\psi^{(m,f)}]$ is represented by the subgroup $\ker(\psi^{(m,f)}) = \N_{F_2}$, 
we consider the action of $\F_2$ on $\{1,2, \dots, d\}$ corresponding to the 
homomorphism $\psi^{(m,f)}: \F_2 \to S_d$. 
Since $[\psi^{(m,f)}]$ is Galois the stabilizer $\Stab_{\F_2}(j)$ of $j$ coincides 
with the kernel of $\psi^{(m,f)}$ for every $1 \le j \le d$.
 
Thus $[\psi^{(m,f)}]$ is represented by the same subgroup $\N_{\F_2}$ and the proposition is proved. 
\end{proof}

\subsection{The hierarchy of orbits}
\label{sec:hierarchy}

Combining Theorem \ref{thm:compatible} with the definition of the cofunctor $\A^{\bbQ}$ (see \eqref{A-bbQ}), 
we get the following statement:
\begin{cor}  
\label{cor:Ihara-OK}
The diagram of (co)functors 
\begin{equation}
\label{A-bbQ-A-A-sh-diag}
\begin{tikzpicture}
\matrix (m) [matrix of math nodes, row sep=2.3em, column sep=2.3em]
{G_{\bbQ, \NFI}  & \GTh_{\NFI} &  \GTSh^{\hs} \\
 ~ & \Dessin & ~  \\};
\path[->, font=\scriptsize]
(m-1-2) edge node[above] {$\PR$} (m-1-3)
 edge node[left] {$\A$} (m-2-2)
(m-1-1) edge node[above] {$\Ih$} (m-1-2)
edge node[left] {$\A^{\bbQ}~~$} (m-2-2)
(m-1-3) edge node[right] {$~~\A^{sh}$} (m-2-2);  
\end{tikzpicture}
\end{equation}
commutes ``on the nose''. \qed
\end{cor}   
  
\bigskip

Let $\K, \N \in \NFI_{\PB_4}(\B_4)$ and $\K \le \N$. 
Recall \cite[Section 3.2]{GTshadows} that, if a pair $(m,f) \in \bbZ \times \F_2$ represents 
a charming $\GT$-shadow with the target $\K$, then the same pair $(m,f)$ also 
represents a charming $\GT$-shadow with the target $\N$. Hence, for every 
pair $\K, \N \in \NFI_{\PB_4}(\B_4)$ with $\K \le \N$, we have a natural map 
\begin{equation}
\label{GT-K-to-GT-N}
\P_{\K, \N} : \GT^{\hs}(\K) \to \GT^{\hs}(\N).
\end{equation}

Let $\psi$ be a homomorphism $\F_2 \to S_d$ that represents a child's drawing 
subordinate to $\N$. Since $\K \le \N$, $[\psi]$ is also 
subordinate to $\K$. 

Let $[m,f]\in \GT^{\hs}(\K)$. Since $[\psi]^{[m,f]}$ depends only on the residue class of $m$ modulo 
$\lcm(\ord(x \ker(\psi)), \ord(y \ker(\psi)))$ and the coset $f \ker(\psi) \in \F_2/ \ker(\psi)$, we have 
\begin{equation}
\label{N-K-OK}
[\psi]^{[m,f]} = [\psi]^{\P_{\K, \N}([m,f])}.
\end{equation}

Combining this observation with Corollary \eqref{cor:Ihara-OK} we get the following statement:
\begin{cor}  
\label{cor:hierarchy}
Let $\K, \N$ be elements of $\NFI_{\PB_4}(\B_4)$ that dominate 
a child's drawing $\cD$ and $\K \le \N$.  Then we have the following hierarchy of orbits: 
\begin{equation}
\label{hierarchy}
\GT^{\hs}(\N)(\cD) \supset \GT^{\hs}(\K)(\cD) \supset \GTh(\cD) \supset G_{\bbQ}(\cD).  
\end{equation}
\qed
\end{cor}   
\begin{remark}  
\label{rem:is-it-proper}
There may be examples of pairs $\K,\N \in \NFI_{\PB_4}(\B_4)$ with $\K \le \N$
such that the natural map $\P_{\K, \N}$ \eqref{GT-K-to-GT-N} is not onto. 
So, in principle, there may be examples of  $\K, \N \in \NFI_{\PB_4}(\B_4)$ with $\K \le \N$
and $\cD \in \Dessin(\N)$ for which the inclusion 
$$
\GT^{\hs}(\N)(\cD) \supset \GT^{\hs}(\K)(\cD) 
$$
is proper. 
At the time of writing, the author did not find any examples of this kind. In fact, for all examples in which 
both orbits $\GT^{\hs}(\N)(\cD)$ and $G_{\bbQ}(\cD)$ can be computed, we have
$\GT^{\hs}(\N)(\cD) = G_{\bbQ}(\cD)$.
\end{remark}
\begin{remark}  
\label{rem:some-day-some-way}
Since the orbit $\GTh(\cD)$ is finite and $\GTh$ is the limit 
of the functor that sends $\N \in \NFI^{isolated}_{\PB_4}(\B_4)$ to the finite group 
$\GT^{\hs}(\N)$ (see \cite[Theorem 3.8]{GTshadows}), for every child's drawing $\cD$, 
there exists $\N \in \NFI^{isolated}_{\PB_4}(\B_4)$ such that 
\begin{itemize}

\item $\cD \in \Dessin(\N)$, i.e. $\N$ dominates $\cD$, and 

\item $\GTh(\cD) = \GT^{\hs}(\N)(\cD)$. 

\end{itemize}
See also, the Corollary and the Problem on page 227 of \cite{HS}.  
\end{remark}

\bigskip

Let $\cD$ be a child's drawing and $H_{\cD}$ be the stabilizer of $\cD$ in $G_{\bbQ}$.  
Recall that the \e{field of moduli} of a child's drawing $\cD$ is the (finite) extension 
$M_{\cD} \supset \bbQ$ corresponding to the (closed) subgroup $H_{\cD} \le G_{\bbQ}$ 
(via the Galois correspondence). 
Since $[M_{\cD} : \bbQ ] =  |G_{\bbQ} : H_{\cD}|$, the 
orbit-stabilizer theorem and Corollary \ref{cor:hierarchy} imply
the following version of \cite[Proposition 14]{HS}:
\begin{cor}  
\label{cor:field-of-moduli}
For every child's drawing $\cD$ and every $\N \in  \NFI_{\PB_4}(\B_4)$ that dominates 
$\cD$, we have 
\begin{equation}
\label{bound-degree}
[M_{\cD} : \bbQ ] \le |\GT^{\hs}(\N)(\cD)|.
\end{equation}
\end{cor}  

Let us also recall \cite[Proposition 2.5]{C-Harbater-Hurwitz}, \cite[Corollary on page 2]{JV-errata} 
that, if the child's drawing $\cD$ is Galois then it admits
a Belyi pair $(X, \ga)$ defined over its field of moduli $M_{\cD}$.  Combining this observation 
with Corollary \eqref{cor:hierarchy}, we conclude that if $\cD$ is a Galois child's drawing 
and $\GT^{\hs}(\N)(\cD)$ is a singleton (for some $\N$ that dominates $\cD$), then
$\cD$ admits a Belyi pair $(X, \ga)$ defined over $\bbQ$. 

A large supply of examples of such Galois child's drawings comes from isolated 
elements of $\NFI_{\PB_4}(\B_4)$. Indeed, combining \cite[Proposition 2.5]{C-Harbater-Hurwitz}, 
\cite[Corollary on page 2]{JV-errata} with Proposition \ref{prop:N-isolated-NF2-great} and  
Corollary \eqref{cor:hierarchy}, we get the following statement: 
\begin{cor}  
\label{cor:Galois-dessin-overQ}
For every isolated element $\N \in \NFI_{\PB_4}(\B_4)$, the child's drawing represented by 
the subgroup $\N_{\F_2} \unlhd \F_2$ admits a Belyi pair defined over $\bbQ$. \qed
\end{cor}  
\begin{remark}  
\label{rem:Galois-dessin-overQ}
Section 4 of \cite{GTshadows} presents the basic information about 35 selected elements 
\begin{equation}
\label{35-cool-guys}
\N^{(0)}, \N^{(1)}, \dots, \N^{(34)} 
\end{equation}
of the poset $\NFI_{\PB_4}(\B_4)$. More information about these elements can be found in \cite{Package-GT}. 
According to Table 1 (on page 40) in \cite{GTshadows}, $27$ of these $35$ elements are isolated. 
Due to Corollary \ref{cor:Galois-dessin-overQ}, for every isolated element $\N$ in list \eqref{35-cool-guys}, 
the child's drawing represented by $\N_{\F_2}$ admits a Belyi pair defined over $\bbQ$.
Moreover, for many of these isolated elements, the degrees of the corresponding child's drawings 
are quite large. For example, the degree of the child's drawing represented by $\N^{(34)}_{\F_2}$ is 
$20575296 = 2^6 \cdot 3^8 \cdot 7^2$.
\end{remark}
\begin{remark}  
\label{rem:better-bound}
Proposition \ref{prop:N-isolated-NF2-great}, Remark \ref{rem:Galois-dessin-overQ}
and examples considered in Section \ref{sec:examples} indicate that the size of the orbit 
$\GT^{\hs}(\N)(\cD)$ (for some $\N$ dominating a child's drawing $\cD$)
may be significantly smaller than the number of the child's drawings with
the same passport as $\cD$. Thus the bound on the degree of the field of moduli 
\eqref{bound-degree} has more practical value than the one given 
in \cite[Proposition 7.1]{Voight-Belyi}.
\end{remark}

\subsection{The monodromy group and the passport are invariant with respect to 
the action of $\GT$-shadows}
\label{sec:invar}

It is easy to see that the degree and the monodromy group of a child's drawing 
are invariant with respect to the action of $\GT$-shadows. (See, for example, \eqref{the-same-mon-group}
in the proof of Proposition \ref{prop:Galois-to-Galois}).

Let us prove that the passport of a child's drawing is invariant with respect to the action of 
charming $\GT$-shadows.
\begin{thm}  
\label{thm:passport-invar}
Let $\N \in \NFI_{\PB_4}(\B_4)$ and $\psi$ be a homomorphism $ \F_2 \to S_d$
that represents $[\psi] \in \Dessin(\N)$. Then, for every charming $\GT$-shadow 
$[m, f]$ in $\GT(\N)$, the child's drawing $[\psi]^{[m,f]}$ has the same passport as $[\psi]$.   
\end{thm}  
\begin{proof} We set 
$$
x:=x_{12}, \qquad y:= x_{23}, \qquad
z:= (xy)^{-1}, \qquad w:= (yx)^{-1}\,,
$$ 
and 
$$
g_x := \psi(x), \qquad g_y:= \psi(y), \qquad 
g_z := \psi(z).
$$

Due to \eqref{def-the-action}, the child's drawing $[\psi]^{[m,f]}$
is represented by the triple:
\begin{equation}
\label{new-triple}
\big( g_x^{2m+1},~ \psi(f)^{-1} g_y^{2m+1} \psi(f), ~ \psi(f)^{-1} g_y^{-2m-1} \psi(f) g_x^{-2m-1} \big).  
\end{equation}

The passport of the child's drawing $[\psi]$ is the triple of partitions 
$(\ct(g_x), \ct(g_y), \ct(g_z))$. 

Thus our goal is to show that 
\begin{equation}
\label{ct-x-ct-y-OK}
\ct(g_x^{2m+1}) = \ct(g_x), 
\qquad 
\ct \big( \psi(f)^{-1} g_y^{2m+1} \psi(f) \big) = \ct(g_y), 
\end{equation}
and 
\begin{equation}
\label{ct-z-OK}
\ct\big( \psi(f)^{-1} g_y^{-2m-1} \psi(f) g_x^{-2m-1} \big) = \ct(g_z)
\end{equation}

It is clear that the second equation in \eqref{ct-x-ct-y-OK} is equivalent 
to $\ct (g_y^{2m+1} ) = \ct(g_y)$. Thus equations \eqref{ct-x-ct-y-OK} are consequences 
of the following simple fact about permutations: if 
$$
\gcd(q, \ord(h)) = 1,
$$
then the permutations $h^q$ and $h$ have the same cycle structure. 

The integer $2m+1$ is coprime with the orders of $g_x := \psi(x)$ and $g_y := \psi(y)$ because 
$2m+1$ is coprime with the orders of $x_{12} \N_{\F_2}$ and $x_{23} \N_{\F_2}$.

The proof of equation \eqref{ct-z-OK} requires more work. 

Since the $\GT$-shadow $[m,f]$ is charming, we may assume, without loss of generality, 
that $f \in [\F_2, \F_2]$. Hence Proposition \ref{prop:simple-hexa} from Appendix 
\ref{app:simple-hexa} implies that the pair 
$(m,f)$ satisfies relations \eqref{H-I} and \eqref{H-II}. 
 
Conjugating \eqref{H-II} by $x^{-m}$, we get 
\begin{equation}
\label{H-IIa}
f(z,x)  z^m f(y,z) y^m f(x,y) x^m \in \N_{\F_2}\,.
\end{equation} 
Furthermore, conjugating \eqref{H-IIa} by $\te := \si_1 \si_2 \si_1$, we get 
\begin{equation}
\label{H-IIb}
f(w,y)  w^m f(x,w) x^m f(y,x) y^m \in \N_{\F_2}\,.
\end{equation} 

Equation \eqref{H-II} implies that 
\begin{equation}
\label{H-IIc}
y^m f(x,y)\,  \N_{\F_2} = f(y,z)^{-1} z^{-m} f(z,x)^{-1} x^{-m} \N_{\F_2}
\end{equation}

Equation \eqref{H-IIb} implies that
\begin{equation}
\label{H-IIbb}
 f(y,x) y^m \, \N_{\F_2} = x^{-m} f(x,w)^{-1} w^{-m} f(w,y)^{-1} \, \N_{\F_2} 
\end{equation} 

To prove equation \eqref{ct-z-OK}, we need to show that the permutations 
\begin{equation}
\label{key-permutations}
\psi(xy) \qquad  \txt{and} \qquad \psi(x^{2m+1} f^{-1} y^{2m+1} f) 
\end{equation}
have the same cycle structure.  

We have 
$$
x^{2m+1} f^{-1} y^{2m+1} f \, \N_{\F_2} = 
x^{2m+1} f(y,x) y^{2m+1} f(x,y) \, \N_{\F_2} = 
x^{2m+1} (f(y,x)y^m) y (y^m f(x,y)) \, \N_{\F_2}\,. 
$$

Applying \eqref{H-IIc} to $y^m f(x,y)$ and \eqref{H-IIbb} to $f(y,x)y^m$ we 
get\footnote{In these calculations, $\sim$ means ``conjugate in $\F_2 /\N_{\F_2}$''.} 
$$
x^{2m+1} f^{-1} y^{2m+1} f \, \N_{\F_2} = 
x^{2m+1}  x^{-m} f(x,w)^{-1} w^{-m} f(w,y)^{-1}  \, y \, 
f(y,z)^{-1} z^{-m} f(z,x)^{-1} x^{-m}  \, \N_{\F_2}
$$
$$
\sim x f(x,w)^{-1} w^{-m} f(w,y)^{-1}  \, y \, 
f(y,z)^{-1} z^{-m} f(z,x)^{-1}  \, \N_{\F_2} 
$$
$$
=  x f(w,x) w^{-m} f(y,w)  \, y \, 
f(z,y) z^{-m} f(x,z) \, \N_{\F_2}\,.
$$

Applying the obvious relations $x w x^{-1} = z$, $y z y^{-1} = w$ to 
$ x f(w,x) w^{-m}$, $f(y,w) y$, respectively, and using \eqref{f-y-z-z-y-f}
and \eqref{f-z-x-x-z-f} from Appendix \ref{app:simple-hexa},
we see that, up to conjugation in $\F_2/ \N_{\F_2}$, 
$$
x^{2m+1} f^{-1} y^{2m+1} f \, \N_{\F_2}  ~\sim~ 
 f(z,x) z^{-m} \, x y \, f(y,z)  
f(z,y) z^{-m} f(x,z) \, \N_{\F_2} = 
$$
$$
f(z,x) z^{-m} \, x y \,  z^{-m} f(x,z) \, \N_{\F_2} ~\sim~
 (xy)^{2m+1} \N_{\F_2}. 
$$
Hence the permutations $\psi(x^{2m+1} f^{-1} y^{2m+1} f)$ and 
$\psi(xy)^{2m+1}$ have the same cycle structure. 

To prove that the permutations in \eqref{key-permutations} have the same cycle structure, 
it remains to show that the cycle structure of $\psi(xy)^{2m+1}$ coincides with the cycle 
structure of $\psi(xy)$. It suffices to show that the order of $xy \N_{\PB_3}$ divides 
$N_{\ord}$. 

Since $xy \N_{\PB_3} = z^{-1} \N_{\PB_3}$, we need to show that 
the order of $z \N_{\PB_3}$ divides  $N_{\ord}$. 

Since $x^{N_{\ord}}_{23} \in \N_{\PB_3}$ and 
$\si_1 x_{23} \si_1^{-1} = c x^{-1}_{23} x^{-1}_{12} = c z$, we have 
$$
c^{N_{\ord}}  z^{N_{\ord}} = (cz)^{N_{\ord}} \in \N_{\PB_3} \,.
$$
Combining this observation with $c^{N_{\ord}} \in  \N_{\PB_3}$, we conclude that 
$z^{N_{\ord} } \N_{\PB_3} =1$. Hence the order of $z \N_{\PB_3}$ divides $N_{\ord}$. 

Since the permutations in \eqref{key-permutations} have the same cycle structure, 
\eqref{ct-z-OK} is proved. 
\end{proof}

\section{Abelian child's drawings}
\label{sec:Abelian}

Recall that the monodromy group of a child's drawing $[\psi]$ of degree $d$
is defined up to conjugation in $S_d$. It is clear that the following definition 
does not depend on the choice of the representative of the monodromy group 
of a child's drawing.  
\begin{defi}  
\label{dfn:Abelian-cD}
A child's drawing $\cD$ is called \e{Abelian} if its monodromy group is Abelian. 
\end{defi}  

For example, for every $d \ge 1$, the pair $((1,2,\dots, d), (1,2,\dots, d))$ represents
an Abelian child's drawing of degree $d$.

Let us prove that 
\begin{prop}  
\label{prop:Abelian-is-Galois}
Every Abelian child's drawing $[(g_1, g_2)]$ is Galois. In particular, the order of the monodromy group
 $G:=\lan g_1, g_2 \ran$ of $[(g_1, g_2)]$ coincides with the degree of $[(g_1, g_2)]$.
 \end{prop}  
\begin{proof}
Let $\psi: \F_2 \to S_{d}$ be the group homomorphism corresponding to $(g_1, g_2)$, 
$\K := \ker(\psi)$ and $\H$ be the stabilizer of $1$ in $\F_2$. Note that, since 
the action of $\F_2$ on $\{1,2, \dots, d\}$ is transitive, $d = |G:\H|$.

Since the quotient $\F_2/\K \cong G$ is Abelian, every 
subgroup $\N \le \F_2$ that contains $\K$ is normal in $\F_2$. In particular, $\H$ is normal in 
$\F_2$. 
Thus the child's drawing $[\psi]$ is Galois. 

Since $\K$ is the normal core of $\H$ in $\F_2$ and $\H$ is normal, 
we conclude that $\H = \K$. Thus the order of the monodromy group 
of $[(g_1, g_2)]$ is $|G:\H| = d$. 
\end{proof}

Using the basic properties of group actions, we can prove the following 
useful properties of Abelian child's drawings: 

\begin{prop}  
\label{prop:cycles-same-length}
Let $[(g_1, g_2)]$ be an Abelian child's drawing of degree $d$.  Then 
\begin{itemize}

\item[a)] For every $i \in \{1,2\}$, the permutation $g_i$ is a 
product of disjoint cycles of the same length. 

\item[b)] If $g_1$ (resp. $g_2$) is a cycle of length $d$ then 
$g_2 \in \lan g_1 \ran$ (resp. $g_1 \in  \lan g_2 \ran$). 
 
\end{itemize}
\end{prop}  
\begin{proof} As above, we set $G:=\lan g_1, g_2 \ran$.

For a), it suffices to prove that $g_1$ is a product of disjoint cycles of 
the same length. 

Let $H_1: = \lan g_1 \ran$ and 
$$
\bigsqcup_{t=1}^r O_t
$$ 
be the partition of the set $\{1,2, \dots, d\}$ into the orbits of the action of $H_1$. 
If $r=1$, i.e. $g_1$ is a cycle of length $d$, then there is nothing to prove.
So we consider the case when the number of orbits $r$ is $\ge 2$. 
Our goal is to prove that all orbits have the same size. 

Since $g_2$ commutes with $g_1$, the subgroup $H_2:=\lan g_2 \ran$
acts on the set of orbits $\{O_1, \dots, O_r\}$. Furthermore, since the permutation 
group $G$ is transitive, $H_2$ acts on $\{O_1, \dots, O_r\}$ transitively. 
 
Therefore, for two distinct orbits $O$ and $\ti{O}$, there exists $k \in \bbZ_{\ge 1}$
such that $g_2^k(i) \in \ti{O}$ for every $i \in O$, i.e. $g_2^k \big|_{O}$ is a map 
\begin{equation}
\label{O-to-ti-O}
O \to \ti{O}.
\end{equation}

Since $g_2^{-k}\big|_{\ti{O}}$ is the inverse of \eqref{O-to-ti-O}, we conclude that 
$O$ and $\ti{O}$ have the same size. Thus the first statement is proved. 

The second statement is a particular case of one of the exercises in 
\cite[Section 4.3]{DummitFoote}. 

Here is a proof of this statement. 

Without loss of generality, we may assume that $g_1 = (1,2,\dots, d)$. 

If $g_2(1)=1$ then, for every $k \in \{2,3, \dots, d\}$, we have 
$$
g_2 (k) = g_2 (g_1^{k-1}(1)) = g_1^{k-1} (g_2(1)) = g_1^{k-1}(1) = k. 
$$
Thus, in this case, $g_2 = \id$. 

If $g_2(1)=k > 1$ then the permutation $h:= g_1^{1-k} g_2$ satisfies these two properties: 
\begin{itemize}

\item $h$ commutes with $g_1$ and 

\item $h(1) = 1$. 

\end{itemize}

Using the previous argument, we conclude that $h= \id$ and hence $g_2 \in \lan g_1 \ran$. 
\end{proof}

\begin{remark}  
\label{rem:lengths-of-cycles} 
Assume that we are in the set-up of Proposition \ref{prop:cycles-same-length}
and $\{1,2,\dots, d\}$ partitions into $r \ge 2$ orbits of the $\lan g_1 \ran$-action. 
Although, $\lan g_2 \ran$ acts transitively on the set of these orbits,  in general, 
$g_2$ is not a product of cycles of length $r$.  
For example, the permutations 
$$
g_1 = (1, 2, 3, 4)(5, 6, 7, 8)(9, 10, 11, 12),  
\qquad
g_2 = (1, 6, 12, 3, 8, 10)(2, 7, 9, 4, 5, 11)
$$
generate an Abelian and transitive subgroup of $S_{12}$.
\end{remark}
\begin{remark}  
\label{rem:ramif-indices-Galois}
Since every Abelian child's drawing is Galois, Proposition
\ref{prop:cycles-same-length} can be deduced from the properties 
of Galois branched coverings (see, for example, \cite[Proposition 3.2.10]{Szamuely}). 
Of course, it is strange to use Riemann's existence theorem \cite{RET} for a statement 
that can be proved using basic properties of group actions. 
\end{remark}

\subsection{$\GTSh$-orbit of an Abelian child's drawing is a singleton}
\label{sec:singleton}

Let us prove the following auxiliary statement:
\begin{prop}  
\label{prop:conjugate}
Suppose that $(g_1, g_2) \in S_d \times S_d$ generate an Abelian and transitive subgroup $G \le S_d$. 
Then, for every $r \in \bbZ$ such that $\gcd(r,\ord(g_1)) = \gcd(r,\ord(g_2)) =1$, there exists 
$h \in S_d$ such that 
\begin{equation}
\label{g1r-g2r}
g_1^r = h g_1 h^{-1} \qquad \txt{and} \qquad g_2^r = h g_2 h^{-1}.
\end{equation}
\end{prop}  
\begin{proof} Recall that, due to Proposition \ref{prop:Abelian-is-Galois}, the group $G$ has 
order $d$. The key fact that is used in the proof of the desired statement
is that any $G$-set of size $d$ with a transitive action 
of $G$ is isomorphic to the $G$-set $G$ with the standard $G$-action by multiplication.

Let us show that $r$ is coprime to $\ord(w)$ for every $w \in G$. 
 
Since $r$ is coprime to $\ord(g_1)$ and  $\ord(g_2)$, $r$ is coprime 
to $\ord(g_1^{t_1})$ and  $\ord(g_2^{t_2})$ for all integers $t_1$ and $t_2$. 

Since $G =  \lan g_1, g_2 \ran$ is Abelian, every $w \in G$ can be written in the form 
$$
w = g_1^{t_1} g_2^{t_2}.
$$ 
Let $k_i := \ord(g_i^{t_i})$ and $k$ be the least common multiple of $k_1$ and $k_2$. 
Since $r$ is coprime to $k_1$ and $k_2$, $r$ is also coprime to $k$. 

Since $w^k = \id$, we conclude that $\ord(w) | k$. Thus we proved that 
$r$ is coprime to $\ord(w)$ for every $w \in G$. 
 
Since the group $G = \lan g_1, g_2 \ran$ is Abelian, the formula  
\begin{equation}
\label{theta}
\te(w):= w^r, \qquad w \in G
\end{equation}
defines a group endomorphism $\te : G\to G$.

Moreover, since $r$ is coprime to $\ord(w)$ for every $w \in G$, 
$w^r=\id$ if and only if $w = \id$. 
Hence the homomorphism $\te$ is injective. 
Since the group $G$ is finite, the injectivity of $\te$ implies its 
surjectivity. Thus $\te$ is actually an automorphism of $G$. 

Precomposing the inclusion homomorphism $G \to S_d$ with $\te$, we get 
the new action $\ma$ of $G$ given by the formula
$$
\ma(w)(i) := w^r(i).
$$ 
 
Since $\te$ is an automorphism, the new action of $G$ on 
$\{1,2,\dots, d\}$ is also transitive. Since $G$ has order $d$ and the new action of $G$ is  
transitive, we conclude that the new $G$-set is isomorphic to the original one. 

The proposition is proved.
\end{proof}

Let us use Proposition \ref{prop:conjugate} to prove that the orbit of 
every Abelian child's drawing (with respect to the action of $\GT$-shadows)
is a singleton:

\begin{cor}  
\label{cor:GT-orbit-singleton}
Let $[(g_1, g_2)]$ be an Abelian child's drawing of degree $d$ and 
$\N \in \NFI_{\PB_4}(\B_4)$ such that $[(g_1, g_2)]$ is subordinate to $\N$. 
For every $[m,f] \in \GT(\N)$, we have 
$$
[(g_1, g_2)]^{[m,f]} = [(g_1 , g_2)]. 
$$
\end{cor}  
\begin{proof} Let $\psi: \F_2 \to \lan g_1, g_2 \ran$ be the homomorphism 
that sends $x_{12}$ (resp. $x_{23}$) to $g_1$ (resp. to $g_2$). Furthermore, 
let $h := \psi(f)$.  

We know that $[(g_1, g_2)]^{[m,f]}$ is represented by the pair 
$$
( g_1^{2 m+1}\,,~ h^{-1} g_2^{2 m+1} h ).
$$
Since the group $\lan g_1, g_2 \ran$ is Abelian, $h^{-1} g^{2 m+1}_2 h = g_2^{2 m+1}$.
Hence the child's drawing $[(g_1, g_2)]^{[m,f]}$ is represented by the pair 
$$
( g_1^{2 m+1}\,,~  g_2^{2 m+1} ).
$$

Since $g_1$ (resp. $g_2$) is the image of $x_{12} \N_{\F_2}$ (resp. $x_{23} \N_{\F_2}$) 
and $\ord(x_{12} \N_{\F_2}) = \ord(x_{12} \N_{\PB_3})$,  $\ord(x_{23} \N_{\F_2}) = 
\ord(x_{23} \N_{\PB_3})$, we have 
$\ord(g_1) | \ord(x_{12} \N_{\PB_3})$ and $\ord(g_2) | \ord(x_{23} \N_{\PB_3})$. 

Hence, using the fact that $N_{\ord}$ is a multiple of the orders 
$\ord(x_{12} \N_{\PB_3})$,  $\ord(x_{23} \N_{\PB_3})$ and $2m+1$ is coprime 
with $N_{\ord}$, we conclude that $2m+1$ is coprime with the integers 
$\ord(g_1)$ and $\ord(g_2)$. 

Thus, applying Proposition \ref{prop:conjugate}, we conclude that 
the pairs $( g_1^{2 m+1}\,,~  g_2^{2 m+1} )$ and $(g_1, g_2)$
represent the same child's drawing.
\end{proof}

\bigskip

Combining \cite[Proposition 2.5]{C-Harbater-Hurwitz}, \cite[Corollary on page 2]{JV-errata} 
with Proposition \ref{prop:Abelian-is-Galois} and Corollaries \ref{cor:hierarchy}, 
\ref{cor:GT-orbit-singleton}, we deduce the following statement: 
\begin{cor}  
\label{cor:GQ-orbit-singleton}
Every Abelian child's drawing $\cD$ admits a Belyi pair $(X,\ga)$ defined over $\bbQ$. 
\end{cor}  
\begin{remark}  
\label{rem:Hidalgo}
As far as the author understands, the statement of the above corollary is a consequence of 
\cite[Corollary 3.4]{Debes-Douai} and it was proved directly 
in paper \cite{Hidalgo} by R.A. Hidalgo (see \cite[Corollary 3.5]{Hidalgo}). 
\end{remark}

\section{Examples of $\GTSh^{\hs}$-orbits for (non-Abelian) child's drawings}
\label{sec:examples}

\subsection{An example of degree $6$, genus $0$ and orbit size $2$}

Let us denote by $\cD_{6,0}$ the degree $6$ child's drawing represented by triple
\begin{equation}
\label{cD-6-0}
\big( (1,4,5,2)(3,6), ~(1,6,3,2)(4,5),~ ((1, 3), (2, 4)) \big).
\end{equation}
The passport of $\cD_{6,0}$ is 
$$
((4,2),(4,2), (2,2,1,1))
$$
and its genus is $0$. A Belyi map representing $\cD_{6,0}$ can be found at 
\url{https://beta.lmfdb.org/Belyi/6T10/4.2/4.2/2.2.1.1/a}.

The $G_{\bbQ}$-orbit of $\cD_{6,0}$ has size $2$ and the Galois conjugate of $\cD_{6,0}$ is represented 
by the permutation triple 
\begin{equation}
\label{cD-6-0-conj}
\big( (1, 4, 5, 2)(3, 6), (1, 2, 5, 6)(3, 4),~ (3, 5)(4, 6) \big).
\end{equation}

Using \cite{Package-GT}, we found an element $\N \in \NFI_{\PB_4}(\B_4)$ that dominates 
$\cD_{6,0}$. $\N$ is the kernel of a group homomorphism $\PB_4 \to S_{192}$ stored in 
the file $E\_dde6genus0$ (see \cite[Section 6]{Package-GT}). Here are some basic facts about 
this element of $\NFI_{\PB_4}(\B_4)$:
\begin{itemize}

\item $|\PB_4: \N| = 289,207,845,356,544 = 2^{10} \cdot 3^{24}$;

\item $|\PB_3:  \N_{\PB_3}| = 46656 = 2^6 \cdot 3^6$;

\item $|\F_2:  \N_{\F_2}| =11664 = 2^4 \cdot 3^6$;

\item the order of the commutator subgroup $[\F_2/ \N_{\F_2} ,  \F_2/ \N_{\F_2}] $ is $729 = 3^6$;

\item $N_{\ord} := |\PB_2 :  \N_{\PB_2} | = 4$;

\item $\N$ is isolated and $\GT^{\hs}(\N)$ is a non-Abelian group of order $32 = 2^5$. 

\end{itemize}

\bigskip

The orbit $\GT^{\hs}(\N) (\cD_{6,0})$ coincides with the $G_{\bbQ}$-orbit of $\cD_{6,0}$. 

The interesting feature of $\cD_{6,0}$ is that the $\GT$-shadow $[3, 1_{\F_2}] \in \GT^{\hs}(\N)$  corresponding 
to the complex conjugation acts trivially on $\cD_{6,0}$. A $\GT$-shadow that transforms $\cD_{6,0}$ to its Galois 
conjugate is represented by the pair 
$$
\big( 1, y x y x^2 y^2 x^{-3} y^{-4} \big) \in \bbZ \times [\F_2 , \F_2].
$$
 
In fact, there are $\GT$-shadows that transforms $\cD_{6,0}$ to its Galois conjugate and belong to the kernel of
the virtual cyclotomic character. Here is an example of such a $\GT$-shadow
$$
[0,  y x y^3 x y^2 x^2 y^{-6} x^{-4} ]  \in \GT^{\hs}(\N). 
$$

\subsection{An example of degree $5$, genus $0$ and orbit size $2$}

Let $\cD_{5,0}$ be the child's drawing of degree $5$ represented by the permutation triple 
\begin{equation}
\label{cD-5-0}
\big((1,4,5,2), ~(2,3,5,4),~(1,4)(2,3) \big).
\end{equation}
The passport of $\cD_{5,0}$ is 
$$
((4,1),(4,1), (2,2,1))
$$
and its genus is $0$. A Belyi map representing $\cD_{5,0}$ can be found at 
\url{https://beta.lmfdb.org/Belyi/5T3/4.1/4.1/2.2.1/a}.

The $G_{\bbQ}$-orbit of $\cD_{5,0}$ has size $2$ and the Galois conjugate of $\cD_{5,0}$ is represented 
by the permutation triple 
\begin{equation}
\label{cD-5-0-conj}
\big( (1,2,5,4),~ (1,5,2,3), ~(1,4)(2,3) \big).
\end{equation}

Using \cite{Package-GT}, we found an element $\N \in \NFI_{\PB_4}(\B_4)$ that dominates 
$\cD_{5,0}$. $\N$ is the kernel of a group homomorphism $\PB_4 \to S_{160}$ stored in 
the file $E\_dde5genus0$ (see \cite[Section 6]{Package-GT}). 

Here are some basic facts about this element of $\NFI_{\PB_4}(\B_4)$:
\begin{itemize}

\item $|\PB_4: \N| = 250000000000 = 2^{10} \cdot 5^{12}$;

\item $|\PB_3:  \N_{\PB_3}| = 8000 = 2^6 \cdot 5^3$;

\item $|\F_2:  \N_{\F_2}| = 2000 = 2^4 \cdot 5^3$;

\item the order of the commutator subgroup $[\F_2/ \N_{\F_2} ,  \F_2/ \N_{\F_2}] $ is $125=5^3$;

\item $N_{\ord} := |\PB_2 :  \N_{\PB_2} | = 4$;

\item $\N$ is not settled and the connected component of $\N$ in $\GTSh^{\hs}$ has exactly 
two objects; the size of $\GT^{\hs}(\N)$ is $16$.

\end{itemize}

\bigskip

The orbits $\GT^{\hs}(\N)(\cD_{5,0})$ and $G_{\bbQ}(\cD_{5,0})$ coincide.

\subsection{Examples of child's drawing subordinate to $\N^{(5)}$ and $\N^{(29)}$ from 
list \eqref{35-cool-guys}}

Recall that \cite[Section 4]{GTshadows} presents basic information about $35$ selected 
elements $\N^{(0)}, \dots, \N^{(34)}$ of $\NFI_{\PB_4}(\B_4)$.
In this subsection, we describe child's drawings of degrees $7, 15$ and $18$ subordinate to $\N^{(29)}$ 
and a child's drawing of degree $8$ subordinate to $\N^{(5)}$. 

Here is the basic information about $\N^{(29)}$: 
\begin{itemize}

\item $|\PB_4 : \N^{(29)}| = 2520= 2^3 \cdot 3^2 \cdot 5 \cdot 7$;

\item $|\PB_3 : \N^{(29)}_{\PB_3}| = 136080 = 2^4 \cdot 3^5 \cdot 5 \cdot 7$;

\item $|\F_2 :  \N^{(29)}_{\F_2}| = 45,360= 2^4 \cdot 3^4 \cdot 5 \cdot 7$;

\item $N^{(29)}_{\ord} = 6$;

\item $\N^{(29)}$ is an isolated object of $\NFI_{\PB_4}(\B_4)$ and 
$\GT^{\hs}(\N^{(29)})$ is non-Abelian group of order $48=2^4 \cdot 3$. 

\end{itemize}

The child's drawings represented by the following permutation triples
\begin{equation}
\label{cD-7-0}
\big( (1,2,3)(4,5)(6,7),~(1,5,6)(2,7)(3,4), ~(1,4)(2,6)(3,7,5) \big),
\end{equation}
\vspace{0.01cm}
\begin{equation}
\label{cD-15-4}
\begin{array}{c}
\big( (1, 2, 3, 4, 5, 6)(7, 8, 9, 10, 11, 12)(13, 14, 15), \\[0.3cm]
(1, 2, 6, 12, 9, 15)(3, 7, 13) (4, 11, 14, 5, 8, 10), \\[0.3cm]
(1, 2, 15, 11, 8, 3)(4, 13, 9, 5, 10, 12)(6, 14, 7) \big)
\end{array}
\end{equation}
\vspace{0.1cm}
\begin{equation}
\label{cD-18-4}
\begin{array}{c}
\big( (1, 10, 17, 2, 9, 18)(3, 12, 13, 4, 11, 14)(5, 8, 15, 6, 7, 16), \\[0.3cm]
(1, 16, 11, 2, 15, 12)(3, 18, 7, 4, 17, 8)(5, 14, 9, 6, 13, 10) \\[0.3cm]
(1, 3, 5)(2, 4, 6)(7, 9, 11)(8, 10, 12)(13, 15, 17)(14, 16, 18) \big)
\end{array}
\end{equation}
are subordinate to $\N^{(29)}$.
We denote by $\cD_{7,0}$ (resp. $\cD_{15,4}$, $\cD_{18,4}$) the child's drawing represented by 
the permutation triple in \eqref{cD-7-0} (resp. in \eqref{cD-15-4},  \eqref{cD-18-4}). 

The child's drawings $\cD_{7,0}, \cD_{15,4}$ and $\cD_{18,4}$ have degrees 
$7$, $15$ and $18$, respectively. 

The passport of $\cD_{7,0}$ is $((3,2,2), (3,2,2), (3,2,2))$ and its genus is zero.
A Belyi map that represents $\cD_{7,0}$ can be found at 
\url{https://beta.lmfdb.org/Belyi/7T6/3.2.2/3.2.2/3.2.2/a}.

The passports of $\cD_{15,4}$ and $\cD_{18, 4}$ are
$$
((6,6,3), (6,6,3), (6,6,3)) ~~~\txt{and}~~~
((6, 6, 6), (6, 6, 6), (3, 3, 3, 3, 3, 3))
$$
respectively. Both child's drawings $\cD_{15,4}$ and $\cD_{18, 4}$ have genus $4$.

The orbit $\GT^{\hs}(\N^{(29)})(\cD_{7,0})$ coincides with the $G_{\bbQ}$-orbit of $\cD_{7,0}$.
In fact, using \cite{Package-GT}, one can show that there is only one child's drawing with the passport 
$((3,2,2), (3,2,2), (3,2,2))$. Hence the $G_{\bbQ}$-orbit of $\cD_{7,0}$ and 
the $\GT^{\hs}(\N)$-orbit of $\cD_{7,0}$ (for any $\N \in \NFI_{\PB_4}(\B_4)$ that dominates 
$\cD_{7,0}$) must have size $1$. 

The orbit $\GT^{\hs}(\N^{(29)})(\cD_{15,4})$ has two 
elements\footnote{Please see Appendix \ref{app:package} or the description 
of the command $orbit(~,~)$ on page 25 of \cite{Package-GT}.} 
and the ``conjugate'' of $\cD_{15,4}$ is 
represented by the permutation triple 
\begin{equation}
\label{cD-15-4-conj}
\begin{array}{c}
\big( (1, 6, 5, 4, 3, 2)(7, 12, 11, 10, 9, 8)(13, 15, 14),  \\[0.3cm]
(1, 15, 9, 12, 6, 2)(3, 13, 7)(4, 10, 8, 5, 14, 11), \\[0.3cm]
(1, 6, 2, 7, 10, 14)(3, 11, 9, 4, 8, 15)(5, 12, 13) \big).
\end{array}
\end{equation}

Currently, database \cite{Belyi_database} does not contain Belyi maps of degree $15$. 
However, we can prove that
\begin{claim}  
\label{cl:is-the-Galois-conj}
The $G_{\bbQ}$-orbit of $\cD_{15,4}$ coincides with the orbit $\GT^{\hs}(\N^{(29)})(\cD_{15,4}) 
= \{\cD_{15,4}, \cD^*_{15,4} \}$, where $\cD^*_{15, 4}$ is the child's drawing represented by 
the permutation triple in \eqref{cD-15-4-conj}.
\end{claim}  
\begin{proof}
The image of the complex conjugation in $\GT^{\hs}(\N^{(29)})$ is represented by the pair 
$(-1, 1_{\F_2})$ and 
$$
\cD^*_{15,4} = \cD^{[-1, 1_{\F_2}]}_{15,4}\,.
$$
Hence  $\GT^{\hs}(\N^{(29)}) \subset G_{\bbQ}(\cD_{15,4})$. 
The inclusion $G_{\bbQ}(\cD_{15,4}) \subset \GT^{\hs}(\N^{(29)})$ is a consequence 
of Corollary \ref{cor:hierarchy}.
\end{proof}

\bigskip

We should remark that the $G_{\bbQ}$-orbit of $\cD_{15,4}$ is significantly smaller than the number 
of child's drawings with the passport $((6,6,3), (6,6,3), (6,6,3))$. 
Using \cite{Package-GT}, one can show that the number of child's drawings with the passport 
$((6,6,3), (6,6,3), (6,6,3))$ is $\ge 260$.

\bigskip

The orbit $\GT^{\hs}(\N^{(29)})(\cD_{18,4})$ is a singleton. 
The child's drawing $\cD_{18,4}$ is special because the corresponding (degree 18)
covering of $\CP^1 -\{0,1, \infty\}$ is Galois. In fact, the child's drawing $\cD_{18, 4}$
can be represented by the subgroup of $\F_2$ that corresponds to the 
commutator subgroup $[\F_2/\N_{\F_2},  \F_2/\N_{\F_2}]$. 
(See Remark \ref{rem:conj-finite-index}.)

\bigskip

Here is the basic information about element $\N^{(5)}$ from \eqref{35-cool-guys}: 
\begin{itemize}

\item $|\PB_4 : \N^{(5)}| = 24= 2^3 \cdot 3$;

\item $|\PB_3 : \N^{(5)}_{\PB_3}| = 864 = 2^5 \cdot 3^3$;

\item $|\F_2 :  \N^{(5)}_{\F_2}| = 288 = 2^5 \cdot 3^2$;

\item $N^{(5)}_{\ord} = 6$;

\item $\N^{(5)}$ is an isolated object of $\NFI_{\PB_4}(\B_4)$ and 
the group $\GT^{\hs}(\N^{(5)})$ is isomorphic to $D_6$ (the dihedral
group of order $12$).  

\end{itemize}

\bigskip

The child's drawing $\cD_{8,0}$ represented by the permutation triple
\begin{equation}
\label{cD-8-0}
\big((1, 2, 3)(4, 5, 6),~ (1, 8, 5)(2, 4, 7), ~(1, 3, 7, 4, 6, 8)(2, 5) \big).
\end{equation}
is subordinate to $\N^{(5)}$. 

$\cD_{8,0}$ has genus $0$ and its passport is
$$
\big( (3, 3, 1, 1),~ (3, 3, 1, 1),~ (6, 2) \big).
$$
The number of child's drawings with this passport is $5$. 

The orbit $\GT^{\hs}(\N^{(5)})(\cD_{8,0})$ is a singleton. Hence the $G_{\bbQ}$-orbit of
$\cD_{8,0}$ is also a singleton. 

Since $\N^{(19)} \le \N^{(5)}$, $\N^{(19)}$ (from the list in \eqref{35-cool-guys}) 
also dominates $\cD_{8,0}$. 

The author could not find $\cD_{8,0}$ in \cite{Belyi_database}. However, $\cD_{8,0}$
can be transformed, by the action of $\B_3$, to the child's drawing represented by the Belyi map
\url{https://beta.lmfdb.org/Belyi/8T12/6.2/3.3.1.1/3.3.1.1/a}. For the action of $\B_3$ on 
child's drawings, we refer the reader to \cite[Construction 1.1.17]{Lando-Z}.

\appendix

\section{Charming $\GT$-shadows satisfy the simplified hexagon relations}
\label{app:simple-hexa}

Let us prove the following auxiliary statement: 
\begin{prop}  
\label{prop:H-I-always}
Let $\N \in \NFI_{\PB_4}(\B_4)$. If a pair $(m,f) \in \bbZ \times \F_2$ 
satisfies the hexagon relations \eqref{hexa1}, \eqref{hexa11} (modulo $\N_{\PB_3}$) then 
\begin{equation}
\label{H-I-here}
f(x,y) f(y,x)  \in \N_{\F_2},
\end{equation}
where $\N_{\F_2} := \N_{\PB_3} \cap \F_2$.
\end{prop}
\begin{proof}
Let us consider the element 
\begin{equation}
\label{sig-sig-sig}
\si_1^{2m+1} f^{-1} \si_2^{2m+1} f \si_1^{2m+1} \,\N_{\PB_3}   ~\in~ \B_3/\N_{\PB_3}
\end{equation}

On the one hand, \eqref{hexa1} implies that 
$$
(\si_1^{2m+1} f^{-1} \si_2^{2m+1} f) \si_1^{2m+1} \,\N_{\PB_3} ~ = ~
f^{-1} \si_1 \si_2 c^m \si_1^{-2m} \si_1^{2m+1}  \,\N_{\PB_3} ~=~
$$
$$
f^{-1} \si_1 \si_2 c^m \si_1  \,\N_{\PB_3} ~=~ f^{-1} \te c^m  \,\N_{\PB_3}\,,
$$
where $\te: = \si_1 \si_2 \si_1$. 

On the other hand, using \eqref{hexa11} and the relation $\si_1 \te = \te \si_2$ we get
$$
\si_1^{2m+1} (f^{-1} \si_2^{2m+1} f \si_1^{2m+1}) \,\N_{\PB_3} ~ = ~
\si_1^{2m+1}  \si_2 \si_1 c^{m} \si_2^{-2m} f \,\N_{\PB_3} ~=
$$
$$
\si_1^{2m}  \te c^{m} \si_2^{-2m} f \,\N_{\PB_3} ~=~ \te  \si_2^{2m} c^{m} \si_2^{-2m} f \,\N_{\PB_3}
~=~ \te f c^m \,\N_{\PB_3}\,.
$$

Thus $\te f c^m \,\N_{\PB_3} = f^{-1} \te c^m  \,\N_{\PB_3}$ and hence 
\begin{equation}
\label{te-f-te-inv}
\te f \te^{-1} \,\N_{\PB_3} = f^{-1} \,\N_{\PB_3}\,. 
\end{equation}

Since $\te x_{12} \te^{-1} = x_{23}$ and  $\te x_{23} \te^{-1} = x_{12}$ and $f \in \F_2$, 
relation \eqref{te-f-te-inv} implies \eqref{H-I-here}. 
\end{proof}

\bigskip

If we assume that $f \in [\F_2, \F_2]$, then \eqref{H-I-here} implies additional
useful properties:
\begin{prop}  
\label{prop:y-z-z-x}
Let $\N \in \NFI_{\PB_4}(\B_4)$ and $\N_{\F_2} := \N_{\PB_3} \cap \F_2$. 
If $f \in  [\F_2, \F_2]$ satisfies \eqref{H-I-here}, then 
\begin{equation}
\label{f-y-z-z-y-f}
f(y,z) f(z,y) \in \N_{\F_2}\,,
\end{equation}
and 
\begin{equation}
\label{f-z-x-x-z-f}
f(z,x) f(x,z) \in \N_{\F_2}\,,
\end{equation}
where $z:= y^{-1} x^{-1}$. 
\end{prop}  
\begin{proof} Relation \eqref{H-I-here} obviously implies that 
\begin{equation}
\label{H-I-N-PB-3}
f(x,y) f(y,x)  \in \N_{\PB_3}.
\end{equation}

Since 
$$
\si_1 \si_2 x_{12} (\si_1\si_2)^{-1} = x_{23} 
~~~\txt{and}~~~
\si_1 \si_2 x_{23} (\si_1\si_2)^{-1} = x_{23}^{-1} x_{12}^{-1} c, 
$$
conjugating \eqref{H-I-N-PB-3} by $\si_1 \si_2$ gives us 
\begin{equation}
\label{f-y-z-c}
f(y, zc) f(zc, y)  \in \N_{\PB_3}\,.
\end{equation}

Since $f \in [\F_2, \F_2]$ and $c \in \cZ(\PB_3)$, relation \eqref{f-y-z-c} implies \eqref{f-y-z-z-y-f}. 

Conjugating relation \eqref{f-y-z-c} by $\si_1 \si_2$, we get 
\begin{equation}
\label{f-zc-x-inv}
f(z c,  x ) f(x , z c)  \in \N_{\PB_3}\,.
\end{equation}
 
Since $f(z,x) \in [\F_2, \F_2]$ and $c \in \cZ(\PB_3)$,
relation \eqref{f-zc-x-inv} implies \eqref{f-z-x-x-z-f}. 
\end{proof}

\bigskip

We will now get the versions of relations (4.3) and (4.4) from Section 4 of Drinfeld's foundational 
paper \cite{Drinfeld}:
\begin{prop}  
\label{prop:simple-hexa}
Let $\N \in \NFI_{\PB_4}(\B_4)$ and $(m,f) \in \bbZ \times [\F_2, \F_2]$. The pair 
$(m,f)$ satisfies the hexagon relations \eqref{hexa1}, \eqref{hexa11} (modulo $\N_{\PB_3}$) 
if and only if $(m,f)$ satisfies the relations
\begin{equation}
\label{H-I}
f(x,y) f(y,x)  \in \N_{\F_2},
\end{equation}
\begin{equation}
\label{H-II}
x^m f(z,x)  z^m f(y,z) y^m f(x,y) \in \N_{\F_2},
\end{equation} 
where $\N_{\F_2} := \N_{\PB_3} \cap \F_2$ and $z:= y^{-1} x^{-1}$.
\end{prop}  
\begin{proof}
Recall that
$$
x:= x_{12}, \qquad y:= x_{23}, \qquad z:= y^{-1} x^{-1}, 
\qquad w := x^{-1} y^{-1}.
$$

Using the relations
$$
\si_2^{-1} x_{12} \si_2 = x^{-1}_{23} x^{-1}_{12} c = z c,  \qquad
\si_2^{-1} x_{23} \si_2 = x_{23}
$$
and the properties $c \in \cZ(\PB_3)$, $f^{-1} \in [\F_2, \F_2]$, 
we rewrite the left hand side of \eqref{hexa1} as follows: 
\begin{equation}
\label{hexa1-LHS}
\si_1 x_{12}^m \, f^{-1} \si_2 x_{23}^m f \, \N_{\PB_3} =  
\si_1 \si_2 c^{m}  z^m \, f^{-1}(z,y)  y^m f \, \N_{\PB_3}\,.
\end{equation}

Using the relations
$$
\si_2^{-1} \si_1^{-1} x_{12} \si_1 \si_2 = x^{-1}_{23} x^{-1}_{12} c = z c, \qquad 
\si_2^{-1} \si_1^{-1} x_{23} \si_1 \si_2 = x_{12} = x
$$
and the properties $c \in \cZ(\PB_3)$, $f^{-1} \in [\F_2, \F_2]$, 
we rewrite the right hand side of \eqref{hexa1} as follows: 
\begin{equation}
\label{hexa1-RHS}
f^{-1} \si_1 \si_2 x_{12}^{-m} c^m \, \N_{\PB_3}   = 
\si_1 \si_2 c^{m} f^{-1}(z,x) x^{-m} \, \N_{\PB_3}\,. 
\end{equation}

Combining \eqref{hexa1-LHS} and \eqref{hexa1-RHS}, we see that relation \eqref{hexa1}
is equivalent to 
\begin{equation}
\label{hexa1-easy}
x^m f(z,x) z^m \, f^{-1}(z,y)  y^m f(x,y)  ~\in~  \N_{\F_2}\,.
\end{equation}

Using the relations
$$
\si_2 \si_1 x_{12} \si_1^{-1} \si_2^{-1} = c x_{12}^{-1} x_{23}^{-1} = c w, 
\qquad 
\si_2 \si_1 x_{23} \si_1^{-1} \si_2^{-1} = x_{12} = x
$$
and the properties $c \in \cZ(\PB_3)$, $f \in [\F_2, \F_2]$, 
we rewrite the right hand side of \eqref{hexa11} as follows: 
\begin{equation}
\label{hexa11-RHS}
\si_2 \si_1 x_{23}^{-m} c^m f \, \N_{\PB_3} = 
x^{-m} f(w, x)  \si_2 \si_1 c^m \, \N_{\PB_3}
\end{equation}

Using the relations
$$
\si_2  x_{12} \si_2^{-1} = c w, \qquad 
\si_2  x_{23} \si_2^{-1} = x_{23} = y
$$
and the properties $c \in \cZ(\PB_3)$, $f \in [\F_2, \F_2]$, 
we rewrite the left hand side of \eqref{hexa11} as follows: 
\begin{equation}
\label{hexa11-LHS}
f^{-1} \si_2 x_{23}^m f \si_1 x_{12}^m \,  \N_{\PB_3}  = 
f^{-1} y^m f(w, y) w^m \si_2 \si_1 c^m \,  \N_{\PB_3}\,.
\end{equation}

Combining \eqref{hexa11-LHS} and \eqref{hexa11-RHS}, we see that relation \eqref{hexa11}
is equivalent to 
\begin{equation}
\label{hexa11-well}
y^m f(w,y) w^m f^{-1}(w,x) x^m f^{-1}(x,y) ~\in~  \N_{\PB_3}\,.
\end{equation}

Conjugating \eqref{hexa11-well} by $\te: = \si_1 \si_2 \si_1$, we see that relation \eqref{hexa11}
is equivalent to
\begin{equation}
\label{hexa11-easy}
x^m f(z,x) z^m f^{-1}(z,y) y^m f^{-1}(y,x) ~\in~  \N_{\F_2}\,.
\end{equation}

We can now prove the desired equivalence. Indeed, if a pair $(m,f) \in \bbZ \times [\F_2, \F_2]$ satisfies 
\eqref{hexa1} and \eqref{hexa11} then, due to Proposition \ref{prop:H-I-always}, relation 
\eqref{H-I} holds. 

Moreover, relation \eqref{f-y-z-z-y-f} from 
Proposition \ref{prop:y-z-z-x} implies that
\begin{equation}
\label{f-z-y}
f^{-1}(z,y) \N_{F_2} = f(y,z) \N_{F_2}\,.
\end{equation}
Hence \eqref{hexa1-easy} implies \eqref{H-II}. 

We proved that relations \eqref{hexa1}, \eqref{hexa11} imply relations \eqref{H-I}, \eqref{H-II}. 

Let us now assume that a pair $(m,f) \in \bbZ \times [\F_2, \F_2]$ satisfies relations 
\eqref{H-I} and \eqref{H-II}. Due to Proposition \ref{prop:y-z-z-x}, relation \eqref{f-z-y} holds. 
Using \eqref{H-I},  \eqref{H-II} and \eqref{f-z-y} we deduce relations
\eqref{hexa1-easy} and \eqref{hexa11-easy}. Since we showed above that relation \eqref{hexa1-easy}
(resp. relation \eqref{hexa11-easy}) is equivalent to \eqref{hexa1} (resp. \eqref{hexa11}), we proved
that relations \eqref{H-I}, \eqref{H-II} imply hexagon relations \eqref{hexa1}, \eqref{hexa11}.
\end{proof}

\section{A few words about the package $\GT$}
\label{app:package}
The material presented in Section 5 of this paper depends heavily on the software package 
$\GT$ for working with $\GT$-shadows and their action on child's drawings. 
See \cite{Package-GT} for the detailed documentation of this package. 

The package is written in Python and it uses commands and functions from the library 
SymPy \cite{SymPy}.  Child's drawings are represented by permutation pairs. 
For a finitely generated group $G \cong \F_n/\K$ (e.g. $\B_4$, $\PB_4$, $\B_3$, $\PB_3$, $\F_2$), 
finite index normal subgroups of $G$ are represented by group homomorphisms $\psi$ from 
$\F_n$ to symmetric groups that satisfy the property $\K \subset \ker(\psi)$.  Since  $\K \subset \ker(\psi)$, 
$\ker(\psi)$ corresponds to exactly one finite index normal subgroup of $G \cong \F_n/\K$
via the correspondence theorem. 

Let $G$ be a finitely generated group, $\psi$ be a group homomorphism from $G$ to $S_q$
and $\N := \ker(\psi)$. Then the quotient group $G/\N$ is identified with the permutation 
group $\psi(G) \le S_q$ and permutations in $\psi(G)$ represent elements of $G/\N$. 
 
Two mathematical statements used in the package stand out: 

\begin{itemize}

\item given group homomorphisms $\psi_1 : \F_n \to S_{d_1}$,  $\psi_2 : \F_n \to S_{d_1}$,  we have 
$$
\ker(\psi_1 \times \psi_2) = \ker(\psi_1) \cap \ker(\psi_2);
$$

\item the cyclotomic character $\chi : G_{\bbQ} \to \Zhat^{\times}$ is surjective. 
 
\end{itemize}

The first statement is often used to check whether two finite index normal subgroups 
(of a finitely generated group) coincide or not. The second statement was used for 
testing the package.  

A simple trick (related to managing memory) was used to generate
words (i.e. elements of $\F_2$ or elements of $[\F_2, \F_2]$) that represent distinct 
elements of $\F_2/ \N_{\F_2}$ or distinct elements of the commutator subgroup 
$[\F_2/ \N_{\F_2}, \F_2/ \N_{\F_2}]$: instead of storing actual instances of the class {\it sympy.combinatorics.permutations.Permutation}, 
we stored only their labels (called ``ranks''). 

We should mention that some commands of the package may be 
time consuming and certain commands have limitations due to computer memory.
Typically, there is an option of timing the commands that may be time consuming.

When we run the main file $PaB.py$, a computer creates the list 
$listE$ of compatible equivalence relations on $\PaB^{\le 4}$ corresponding to $35$ distinct elements
\begin{equation}
\label{listN-app}
\N^{(0)},~ \N^{(1)},~ \dots,~ \N^{(33)}, ~\N^{(34)} 
\end{equation}
of the poset $\NFI_{\PB_4}(\B_4)$.  Table 1, on page 11 in \cite{Package-GT}, shows basic information 
about these compatible equivalence relations. For every $0 \le i \le 34$, the quotient group $\F_2/ \N^{(i)}_{\F_2}$ 
is non-Abelian. 

To produce the 35 elements in \eqref{listN-app}, 
we used a generator of all conjugacy classes of 
group homomorphisms $\psi : \B_4 \to S_d$. It is clear that, for every 
homomorphisms $\psi :\B_4 \to S_d$, the subgroup  
$\ker(\psi) \cap \PB_4 \le \PB_4$ belongs to the poset $\NFI_{\PB_4}(\B_4)$. 
For more details, we refer the reader to \cite[Subsection 4.1.1]{Package-GT}. 

Here are selected things one could do with the groupoid 
of $\GT$-shadows using the package:  
\begin{itemize}

\item Given an element $\N \in \NFI_{\PB_4}(\B_4)$, one could generate 
all $\GT$-shadows with the target $\N$ and all charming $\GT$-shadows with the target $\N$. 
For example, all $\GT$-shadows with the target $\N^{(i)}$ are found for every $0 \le i \le 32$ and 
all charming $\GT$-shadows with the target $\N^{(i)}$ are found for every $0 \le i \le 34$. 

\item  One can test whether an element $\N \in \NFI_{\PB_4}(\B_4)$ is an isolated 
object of the groupoid $\GTSh^{\hs}$ (i.e. whether $\GTSh^{\hs}(\N, \N) = \GT^{\hs}(\N)$.)
For example, $27$ of the $35$ elements in \eqref{listN-app} are isolated.  

\item One can compose $\GT$-shadows (if they can be composed) and one can 
compute the inverse\footnote{Note that computing the inverse of a charming $\GT$-shadow 
may be time consuming.} of a charming $\GT$-shadow.

\item Given $\K, \N \in  \NFI_{\PB_4}(\B_4)$ with  
$\K \le \N$ and a $\GT$-shadow $[m, f] \in \GT(\N)$, one can determine 
whether $[m, f]$ survives into $\K$, i.e. whether $[m, f]$ belongs to the image of the natural map 
$\GT(\K) \to \GT(\N)$.  
 
\item Given an isolated object $\N$ of the groupoid $\GTSh^{\hs}$ and the complete 
list of charming $\GT$-shadows with the target $\N$, one can produce a permutation group isomorphic 
to the group $\GT^{\hs}(\N) = \GTSh^{\hs}(\N, \N)$.

\item Given $\K, \N \in  \NFI_{\PB_4}(\B_4)$ with  $\K \le \N$, one 
can look for fake charming $\GT$-shadows with the target $\N$. 
See \cite[Corollary 3.13]{GTshadows}.

\item Given $\N \in  \NFI_{\PB_4}(\B_4)$, one can check whether $\N$ satisfies 
the strong Furusho property or the weak Furusho property. See \cite[Section 5.2]{Package-GT} or 
\cite[Section 4.3]{GTshadows}. 

\end{itemize}

\bigskip

Here are selected things one could do with child's drawings using the package:  

\begin{itemize}

\item Given a child's drawing $\cD$, one can compute its passport, its genus and its monodromy 
group; one can also check whether $\cD$ is Galois or not. 

\item Given child's drawings $\cD$ and $\cD'$, one can test whether $\cD = \cD'$ or not.

\item Given a triple $\tau$ of partitions of $d \in \bbZ_{> 1}$, one can generate 
all child's drawings (if any) whose passport is $\tau$. 

\item Given an integer $d > 1$, one can generate all child's drawings of degree $d$.

\item Given a child's drawing $\cD$ and $\N \in \NFI_{\PB_4}(\B_4)$, one can 
check whether $\cD$ is subordinate to $\N$.

\item Given a child's drawing $\cD$ subordinate to $\N \in \NFI_{\PB_4}(\B_4)$ 
and a $\GT$-shadow $[m, f]$ with the target $\N$, 
one can compute the result $\cD^{[m, f]}$ of the action of $[m, f]$ on $\cD$. 

\item Given a child's drawing $\cD$ subordinate to $\N$ and 
the complete list of $\GT$-shadows (resp. charming $\GT$-shadows) with the target $\N$, 
one can compute the orbit $\GT(\N)(\cD)$ (resp. the orbit $\GT^{\hs}(\N)(\cD)$).

\item Given an element $\N \in \NFI_{\PB_4}(\B_4)$ and a positive integer $d$, 
one can generate all (if any) child's drawings subordinate to $\N$.  

\item Given a child's drawing $\cD$, one can find an element $\N \in \NFI_{\PB_4}(\B_4)$
that dominates $\cD$. 

\end{itemize}


\bigskip
\bigskip

\noindent\textsc{Department of Mathematics,
Temple University, \\
Wachman Hall Rm. 638\\
1805 N. Broad St.,\\
Philadelphia PA, 19122 USA \\
\emph{E-mail address:} {\bf vald@temple.edu}}

\end{document}